\documentclass{scrartcl}
\usepackage[utf8]{inputenc}
\usepackage{amsfonts,amsmath,amsthm}
\usepackage{graphicx}
\usepackage{amsmath}
\usepackage{mathrsfs}
\usepackage{multicol}
\usepackage{lipsum}
\usepackage{slashbox}
\usepackage{mwe}
\usepackage{authblk}

\DeclareMathOperator{\erf}{erf}

\DeclareMathOperator{\csch}{csch}
\DeclareMathOperator{\sech}{sech}
\DeclareMathOperator{\Law}{Law}
\DeclareMathOperator{\Var}{Var}
\DeclareMathOperator{\SEQ}{SEQ}

\newtheorem{lemma}{Lemma}
\newtheorem{corollary}{Corollary}
\newtheorem{proposition}{Proposition}
\newtheorem{remark}{Remark}
\newtheorem{theorem}{Theorem}
\newtheorem{definition}{Definition}

\begin{document}

\title{On binomial order avalanches}
\author[1]{Friedrich Hubalek\thanks{Financial and Actuarial Mathematics, Vienna University of Technology,
Wiedner Hauptstr. 8 / 105-01 \& -05 FAM,
1040 Vienna, Austria.
(fhubalek@fam.tuwien.ac.at)}}
\author[1]{Dragana Radoji\v ci\' c\thanks{Financial and Actuarial Mathematics, Vienna University of Technology,
Wiedner Hauptstr. 8 / 105-01 \& -05 FAM,
1040 Vienna, Austria.
(dragana@fam.tuwien.ac.at)}}

\affil[1]{TU Wien,  Austria}

\maketitle

\abstract{This paper introduces a discrete limit order book model where new orders
are placed with a fixed displacement from the mid-price. Further, the trade event occurs whenever the mid-price hits the price level on which there is some volume.  Therefore, the dynamics of the limit order book model leads to two trading mechanisms, namely Type~I trade and Type~II trade.   A Type~I trade takes place whenever the price maximum increases, while a  Type~II trade occurs if the price drops by~$\mu$ or more and
then increases by~$\mu$ again. 
 Our focus is mainly on the distribution of order avalanches length, and by an avalanche length we consider a series  of order executions where the  length of periods with no trade event cannot exceed $\varepsilon$. }

\section{Introduction}
\label{sec:intro}

The emergence of automated high frequency trading in the absence of a Tobin
tax, i.e. a microscopic but nevertheless non-zero tax on financial
transactions, has put financial markets recently at times into high distress.
This is highlighted by the large scale occurrences of \emph{flash crashes},
which are fast and deep falls in security prices, together with an as rapid
upward movement to previous levels.  
Among the most notable of those events, it is worth to mention the one
occurring on 6th of May 2010 when the Dow Jones incurred a loss of 9.5\% just
to recover within 15 minutes (see \cite{kirilenko2017flash}); the April 2013 flash crash which removed about
136 billions of dollars from the S\&P 500, only to be regained within minutes;
as well as the Singapore 2013 flash crash which momentarily wiped out about
6.9 billion dollars of capitalization. 
There is an ongoing debate about the
causes and responsibilities of players involved. However, we are not entering
this discussion which is intensely debated and where there are plenty of
sources easily to be found out, but rather study distributional properties of
this phenomenon.

The interest into \emph{order avalanches} is motivated by the
theory of self-organized criticality (SOC), see \cite{SC}. 

Our study is targeting an avalanche object in the specific context of financial
markets.

In recent years, there has been increasing interest in studies of the dynamics of the limit order book, both from the theoretical and practical point of view, see \cite{abergel2016limit}.  Modeling the dynamics of the order book as an interacting  Markovian queueing system is exhibited in \cite{CL}, and the analytical tractability of the model provides the relation between order flow and price dynamics.  

The stochastic auction model in which buy and sell orders came at the auction by following independent renewal processes is studied in \cite{K}, and the number of present orders are modeled by stochastic processes and further limiting distributions are investigated. Delattre et
al. \cite{dLRR} investigate the efficient price that can be employed in practice, and also they developed price statistical estimation using the order flow.

 The authors  in \cite{abergel2013mathematical} investigated the order book model as a multidimensional continuous-time Markov chain, and further using the functional central limit theorem authors showed that the rescaled price process converges to Brownian motion.  In  \cite{BHQ} the authors investigate a stochastic LOB model that converges to a continuous-time limit, and the limits of the buy and sell volume densities are modeled by stochastic partial differential equations coupled with a two-dimensional Brownian motion.  
  A model in which limit order book dynamics depend on the available prices and volume that corresponds to the current prices is introduced and studied in \cite{horst2017weak}. Especially, in \cite{horst2017weak} the authors concluded that when the order size  
 and tick size is converging to zero, by a weak law of large numbers, the volume densities tend to non-linear PDEs, that are coupled with non-linear ODEs, which correspond to the best ask and best bid price.

The structure of the paper is as follows. 
 Section~\ref{SectionModel} introduces a discrete limit order book model and a key quantity of our research, namely avalanche length. Further, in Section~\ref{sec:simplified}  the probability generating function  for the simplified avalanche length is established. Section~\ref{sec:more} contains limit results for the full avalanche length and the generating function for the full avalanche length. The generating function of the time to the first trade, if we start with the initially empty book, is presented in Section~\ref{SectionInitiallyEmpty}.

\section{A simple limit order book model with discrete time and space}\label{SectionModel}
\label{sec:simple}
In \cite{DDW,R,Sp} a simple order book model driven by arithmetic Brownian
motion has been studied. These studies lead to interesting known and new
Brownian path functionals involving Brownian local time and the theory of
Brownian excursions. In the past it has been shown that such results can be
illustrated, found, and often even proven by limit arguments with or from
corresponding results for random walks.

For example, let us refer to
\cite{Ver1979,Tak1995,Tak1999,Csa1994,CR1992,CH2004,LH2007,CM1986,BCP2003} and
in particular \cite{SD1988,CH2003,PW2014}. 

Further (possibly) related
references are \cite{Ald1998,Foe1994,DT1996,Csa1996,LM2007}.

In this section we follow this approach and define a simple order book process
driven by a simple symmetric random walk.
\subsection{The basic setting}
\label{subsec:basic}
We start with a probability space $(\Omega,\mathcal{F},P)$ that carries a
simple symmetric random walk $\{S_{n}:n\geq0\}$. This process is interpreted
as the \emph{mid price}, i.e. the arithmetic average between best bid and ask
price. See also \cite{dLRR} for a general discussion of various price
concepts, or rather \emph{price substitutes} in the context of order book modelling.

For better readability we use mixed index and function notation freely, i.e.
is $S_{k}$ and $S(k)$ means the same.

For simplicity we focus on the the ask-side of the order book and ignore order
cancellations. We fix an integer spread parameter $\mu\geq1$. Further, let us introduce the
ask order volume process $\{V(n,u):n\geq0,u\in\mathbb{Z}\}$, which is a
two-parameter process and $V(n,u)$ denotes the volume of order at time $n$
and price $u$. As a further simplification, we ignore the size of the order
volume, and distinguish only the cases $V(n,u)>0$, if there is an order, or
$V(n,u)=0$, when there is no order in the book.

We assume, that we start from an order book that is initially full, which is
convenient for our analysis. So this means
\begin{equation}
V(0,u)=I_{u\geq0},\quad u\in\mathbb{Z}. \label{V0}%
\end{equation}
Later we will consider an order book that is
initially empty.\footnote{Some exchanges actually clear the order book before
trading starts on a new day, e.g. the Istanbul Stock Exchange, see \cite{VZFR}
for a statistical analysis.}

The dynamic development of the order book is described as follows. If there
there is an order at time $n$ at the level $u=S_{n}$, indicated by
$V(n,S_{n})>0$, it is executed and the corresponding entry is removed from the
order book. Thus at the next time step $V(n+1,S_{n})=0$. This is also true, if
there was no order, indicated by $V(n,S_{n})=0$. Furthermore we assume a new
order will be placed at distance $\mu$ above the price $S_{n}$, thus
$V(n,S_{n}+\mu)=V(n-1,S_{n}+\mu)+1$. To summarize, we have
\begin{equation}\label{dynamics}
V(n,u)=\left\{
\begin{array}
[c]{ll}%
0 & \mbox{if $u=S_{n-1}$}\\
V(n-1,u)+1 & \mbox{if $u=S_{n-1}+\mu$}\\
V(n-1,u) & \mbox{otherwise,}
\end{array}
\right.  \qquad n\geq1,u\in\mathbb{Z}.
\end{equation}

\subsection{Type I and Type II trades, trading times and inter-trading times}
\label{subsec:trades}
A trade event occurs at time $n$ if the mid-price reaches the price level $U$, i.e. $U=S_n$, and there is some volume at  the price level $U$, i.e. if $V(n,S_{n})>0$. 
\begin{definition}

We define trading
times $\{\tau_{i}:i\geq0\}$ and intertrading times $(T_{i})_{i\geq1}$  by
\begin{equation}
\tau_{0}=0,\quad\tau_{i}=\inf\{n>\tau_{i-1}:V(n,S_{n})>0\},\quad T_{i}%
=\tau_{i}-\tau_{i-1},\quad i\geq1.
\end{equation}
\end{definition}
 
It is quite intuitive that typically a trade on the ask-side occurs, if the
midprice moves up. We formalize this as follows. For $i\geq1$ we say
the $i$-th trade is a Type~I trade, if $S(\tau_{i})>S(\tau_{i-1})$. Otherwise
we call the trade a Type~II trade. igure~\ref{FigTypeI} displays an example for a typical
Type~I trade.  \begin{figure}[ptb]
\begin{center}
\includegraphics[width=0.5\textwidth,height=0.37\textheight]{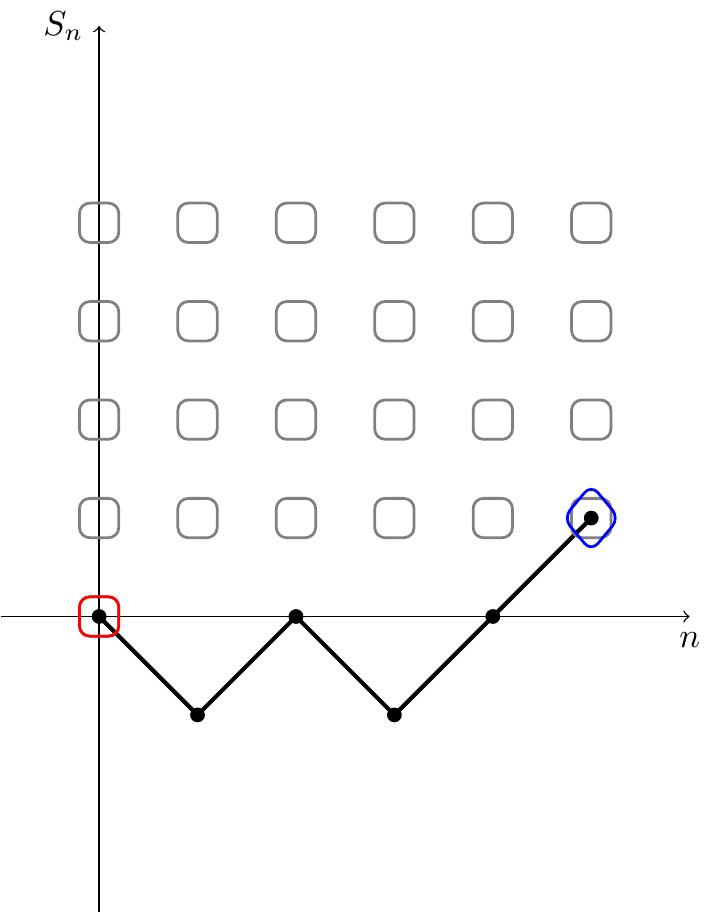}\newline%
\caption{Type I trade (depicted with the blue rhombus) example.}%
\label{FigTypeI}%
\end{center}
\end{figure}A Type II trade occurs, if the price drops by~$\mu$ or more, and
then increases by~$\mu$ or more again. If this happens in a very short time
interval, say of length less than~$\varepsilon>0$, we say, we have a
\emph{flash-crash} trade. Figure~\ref{FigTypeII} displays an example for a typical
 Type~II trade (blue rombs depict Type~I trades, while red rectangles depict Type II trades). Intuitively, Type~II trades, and in particular, flash-crash trades
are less frequent than Type~I trades. Note that the first trade that occurs at the initial position (at the level $u=0$) at the starting time ($n=0$), we define as a Type~II trade. 
\begin{remark}
Note that all pictures included in this paper are generated by the program that simulate the dynamics summarized in equation (\ref{dynamics}).
\end{remark}
A sufficient condition for a Type~I trade to occur is given by the following
easy lemma.

\begin{lemma}\label{EasyI}
If $n\geq1$ is a strict ascending ladder time of the random walk, then there
is a Type~I trade at $n$.
\end{lemma}

\begin{proof}
Suppose $n\geq 1$ is a strict ascending ladder time.
Then, we know that

$S_n>\max(S_0,\ldots,S_{n-1}) \geq 0$
and there was no trade at level $S_n$ or higher before $n$. Since we assume $V(0,u)>0$ for $u\geq0$
it follows $V(n,S_n)>0$, thus there is a trade, and at a price level which is
higher than the level of the last trade.
\end{proof}

\begin{figure}[ptb]
\begin{center}
\includegraphics[width=0.5\textwidth,height=0.4\textheight]{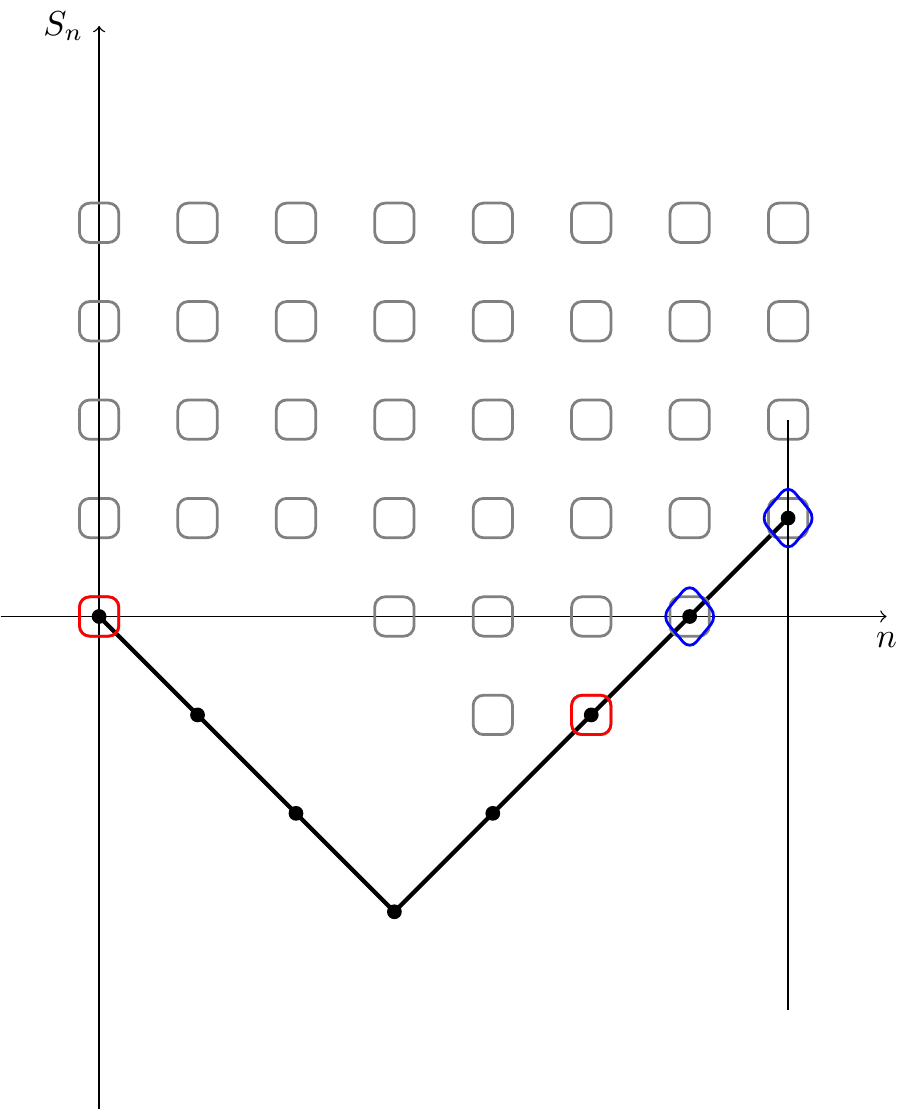}\newline%
\caption{Example of a Type II trade (depicted with the red squares), that is followed by two consecutive Type I trades (depicted with the blue rhombus).}%
\label{FigTypeII}%
\end{center}
\end{figure}

Motivated by this lemma, and also in agreement with \cite{DDW}, we
introduce the \emph{simplified trading times} and the \emph{simplified intratrading times}, that are obtained by ignoring
Type~II trades, by following definition.
\begin{definition}\label{defSimplified}
Define the simplified trading times $\{\rho_{i}:i\geq0\}$ and the
simplified intrading times $(R_{i})_{i\geq1}$  as
\begin{equation}
\rho_{0}=0,\quad\rho_{i}=\inf\{n>\rho_{i-1}:S_{n}>S(\rho_{i-1})\},\quad
R_{i}=\rho_{i}-\rho_{i-1},\quad i\geq1.
\end{equation}
\end{definition}

\begin{figure}[h]
\begin{center}
\includegraphics[width=0.5\textwidth,height=0.4\textheight]{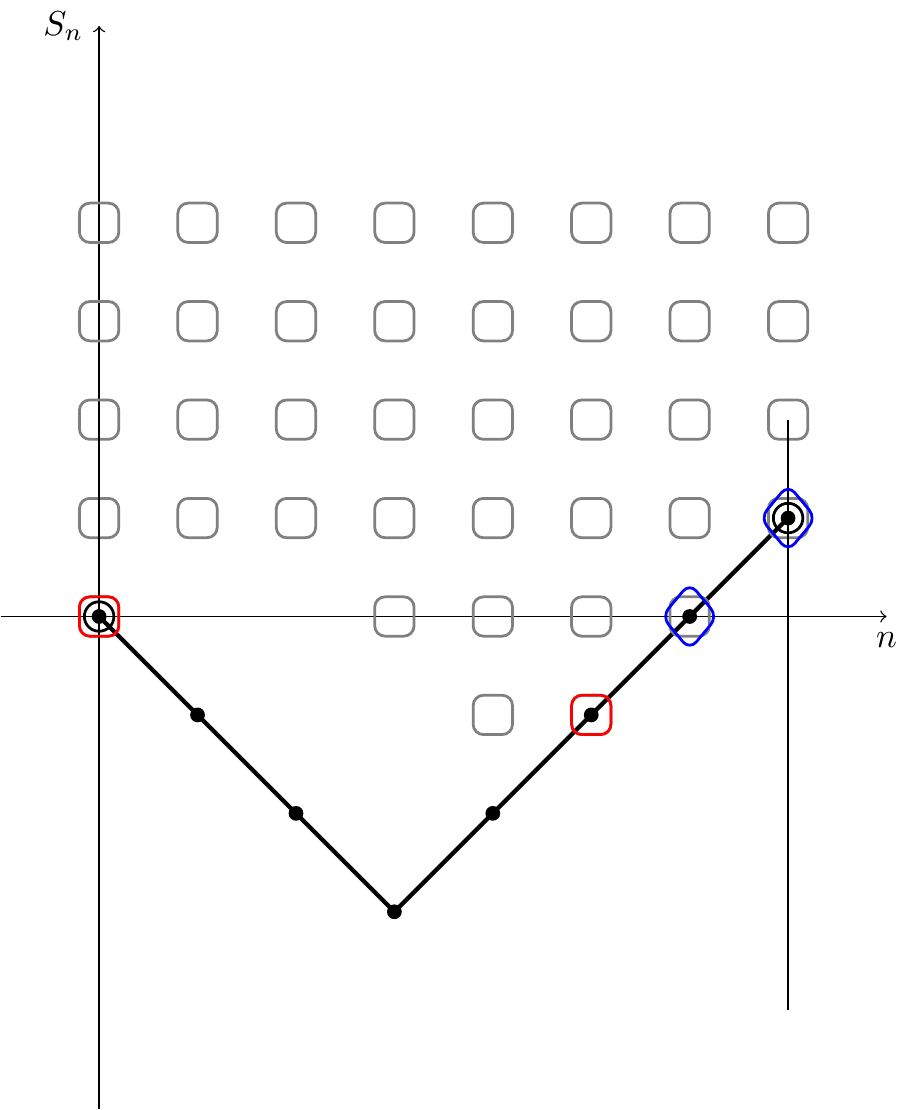}\newline%
\caption{Black circles depict trades that occur at the simplified trading times.}%
\label{FigSimpl}%
\end{center}
\end{figure}

\begin{remark}
Note that $\{\rho_{j}:j\geq0\}\subset \{\tau_{i}:i\geq0\}$ and that with Definition \ref{defSimplified} we are ignoring Type~II trades, but we may also ignore some Type~I trades that
occur after a Type~II trade (see Figure \ref{FigSimpl} for an illustration, in which simplified trading times are depicted by circles). In Figure 3 we can see Type~II trade which is followed with two consecutive Type~I trades, of which first Type~I trade is ignored by Definition \ref{defSimplified}. 
\end{remark}

\subsection{On the best ask price}
\label{subsec:best}
Our next goal is to show that the best ask price can be described by the price
process without reference to the full order volume process. For that purpose
we define the best ask price process ${\{\alpha_{n}\}}_{n\geq 0}$ by
\begin{equation}\label{bestAsk}
\alpha_{0}=0,\quad\alpha_{n+1}=\alpha_{n}+I_{\{\alpha_{n}=S_{n}\}}%
-I_{\{\alpha_{n}=S_{n}+\mu+1\}},\quad n\geq0.
\end{equation}

Equation (\ref{bestAsk}) can be interpreted as follows: if $\alpha_{n}=S_{n}$ the midprice hits the best ask, and therefore the next best ask is higher than the previous one; if $\alpha_{n}=S_{n}+\mu+1$ the order is placed is the previous best ask and the next best ask is lower than the previous one. 

\begin{lemma}
We have
\begin{equation}
S_{n}\leq\alpha_{n}\leq S_{n}+\mu+1,\quad n\geq0.
\end{equation}
\end{lemma}
\begin{proof}
By using induction on $n$ and distinguishing for the step $n\mapsto n+1$ the six cases corresponding
to: $\alpha_n=S_n$ and $S_{n+1}=S_n\pm1$, $\alpha_n=S_n+\mu+1$ and $S_{n+1}=S_n\pm1$, $S_n<\alpha_n<S_n+\mu+1$ and $S_{n+1}=S_n\pm1$.
\end{proof}
\begin{lemma}
We have for all $n\geq0$ that $V(n,u)>0$ iff $u\geq\alpha_{n}$.
\end{lemma}
\begin{proof}
We do again induction on $n\geq0$. For $n=0$ the intial conditions show the claim is true.
For the induction step $n\mapsto n+1$ we distinguish again three cases.
Case~1: Suppose $S_n=\alpha_n$. Then $\alpha_{n+1}=\alpha_n+1$.
By the induction hypothesis $V(n,u)>0$ iff $u\geq\alpha_n$. There is a trade at time $n$ and
thus $V(n+1,S_n)=V(n+1,\alpha_n)=0$. A new order is placed above $V(n+1,S_n+\mu)=V(n,S_n+\mu)+1>0$,
but we had already $V(n,S_n+\mu)>0$. All other positions are unchanged.
Case~2: Suppose $\alpha_n=S_n+\mu+1$. Then $S_n<\alpha_n$ and $V(n,S_n)=0$, thus no trade takes place.
A new order is placed at $S_n+\mu=\alpha_n-1=\alpha_{n+1}$, thus $V(n+1,\alpha_{n+1})>0$. All other positions
are unchanged.
Case~3: Suppose $S_n<\alpha_n<S_n+\mu+1$. Then $\alpha_{n+1}=\alpha_n$. No trade takes place also in this case.
The new order is placed at $S_n+\mu\geq\alpha_n$. All other positions
are unchanged.
\begin{corollary}\label{cor1}
For $i\geq1$ we have 
\begin{equation}
\tau_i=\inf\{n>\tau_{i-1}:S_n=\alpha_n\}.
\end{equation}
\end{corollary}
\end{proof}
\begin{remark}
The trade happens exactly when the mid-price hits the best ask price, i.e. when $\alpha_n=S_n$, $n\geq1$.
\end{remark}

\begin{remark}  In Figure \ref{FigAlpha} red points represent the best ask price process ${\{\alpha_{n}\}}_{n\geq 0}$. 
Figure \ref{FigAlpha} is generated by program simulation and it confirms the aforementioned results for the best ask price process ${\{\alpha_{n}\}}_{n\geq 0}$. 
\end{remark}

\begin{figure}[ptb]
\begin{center}
\includegraphics[width=0.7\textwidth,height=0.4\textheight]{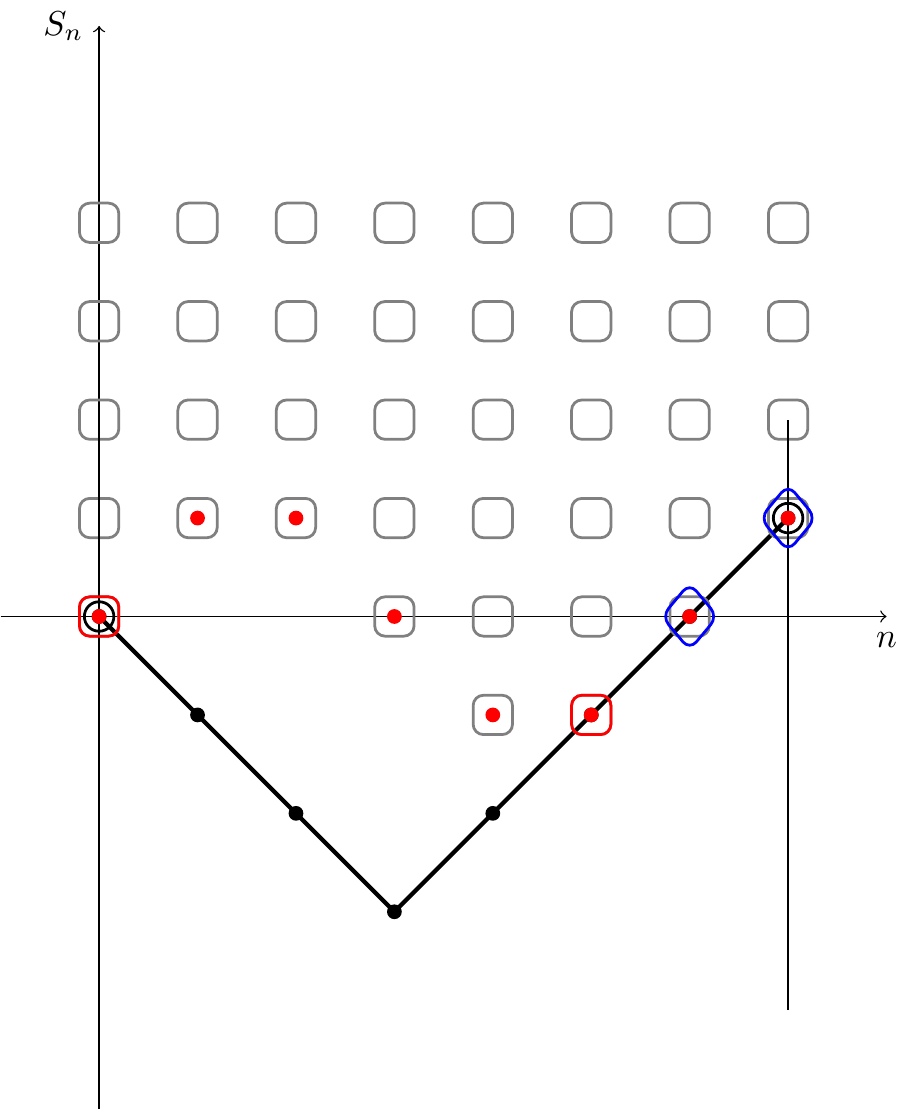}\newline%
\caption{In this Figure red points represent the best ask process}%
\label{FigAlpha}%
\end{center}
\end{figure}

We can now define the \emph{trading excursion process}. We define
it as a process that takes values in the set of simple random paths of finite
length. More precisely, for $n\geq0$ let
\begin{equation}
\mathbf{U}^{(n)}=\left\{  (s_{0},\ldots,s_{n})\in\mathbb{Z}^{n+1}%
:\mbox{$s_0=0$, $|s_j-s_{j-1}|=1$ for $j=1,\ldots,n$}\right\},
\end{equation}
and $\mathbf{U}^{(\infty)}=\bigcup_{n\geq0}\mathbf{U}^{(n)}$. Let
$\mathcal{U}^{(\infty)}$ denote the power set of $U^{(\infty)}$ and define a
discrete measure~$\nu$ by $\nu(\{(s_{0},\ldots,s_{n})\})=2^{-n}$. Then
$(U^{(\infty)},\mathcal{U}^{(\infty)}),\nu)$ becomes a $\sigma$-finite measure
space. Next we define the \emph{trading excursion process} $(e_{i})_{i\geq1}$
by
\begin{equation}
e_{in}=S(\tau_{i-1}+n)-S(\tau_{i-1}),\quad0\leq n\leq T_{i},\quad i\geq1,
\end{equation}
which takes values in $U^{(\infty)}$.
\begin{remark}
For each $i\geq1$ if $e_{iT_{i}}>0$ the Type I trade occurs at time $\tau_{i}$, otherwise if $e_{iT_{i}}\leq0$  the Type II  trade occurs at time $\tau_{i}$.
\end{remark}

\begin{lemma}
The trading excursions $e_{i}, i\geq1$ are iid.
\end{lemma}

\begin{proof}
Since $\tau_i,i\geq0$ are stopping times for the random walk, the claim follows from the Markov property
of the random walk. To give some details, fix $i\geq1$ and consider
the processes
\begin{equation}
S'(n)=S(\tau_{i-1}+n)-S(\tau_{i-1}),\quad
\alpha'(n)=\alpha(\tau_{i-1}+n)-\alpha(\tau_{i-1}),\quad
n\geq0.
\end{equation}
\end{proof}The trading excursion process is uniquely determined by the price
process and conversely, the price process can be reconstructed from the
trading excursion process by glueing the excursions together, similar as in
the classical excursion theory, cf.~\cite[Prop.XII.2.5, P.482]{RY}.

\subsection{Avalanche length}
\label{subsec:avalanche}
Instead of studying the dynamics of the two-parameter order-book process,  following \cite{R} and other earlier work, and motivated by \cite{SC}, we
focus on
a scalar key quantity, the \emph{avalanche length}. When there is a longer
up-movment of the asset price, then a substantial portion of the order book is
\textquotedblright eaten up\textquotedblright\ until the price drops again.
This can perhaps be described as an \textquotedblleft order
avalanche\textquotedblright.

But long up movements are rare for a random walk and when we think of the
corresponding model for continuous time, the probability that Brownian motion
increases on an interval of positive length is zero. Therefore we consider as
an avalanche a period of trade executions, but allow a small window, of size
$\varepsilon$ at most, without trading. Only when there is no trade for a period
of length larger than $\varepsilon$ the avalanche terminates. \begin{figure}[h]
\includegraphics[width=0.9\textwidth,height=0.4\textheight]{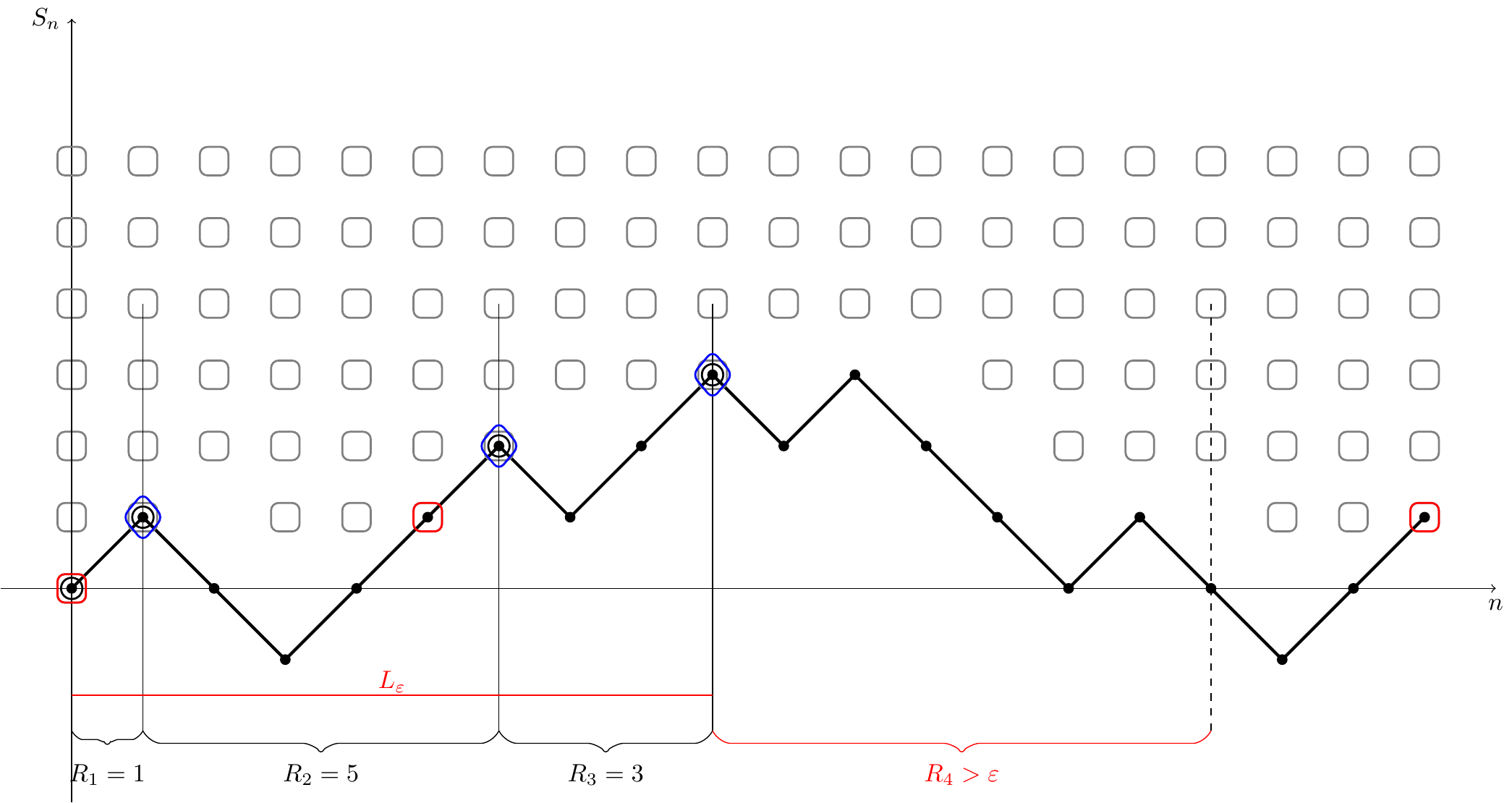}\newline%
\caption{Avalanche example without flash-crash trades when $\varepsilon=7$}%
\label{figsimple}%
\end{figure}Figure~\ref{figsimple} shows an avalanche consisting of a sequence
of typical Type I trades. For the simplified avalanche length we consider only
Type I trades, and completely ignore (do not take into account) Type II trades and even some Type I trades which follow Type II trade (as depicted in the Figure \ref{FigSimpl}).

But an avalanche may continue, even if the price drops, but recovers quickly,
that is, when we have a Type II flash-crash trade. This situation is illustrated in 
Figure~\ref{figfull} for $\varepsilon=8, \mu=2$. \begin{figure}[h]
\includegraphics[width=0.9\textwidth,height=0.4\textheight]{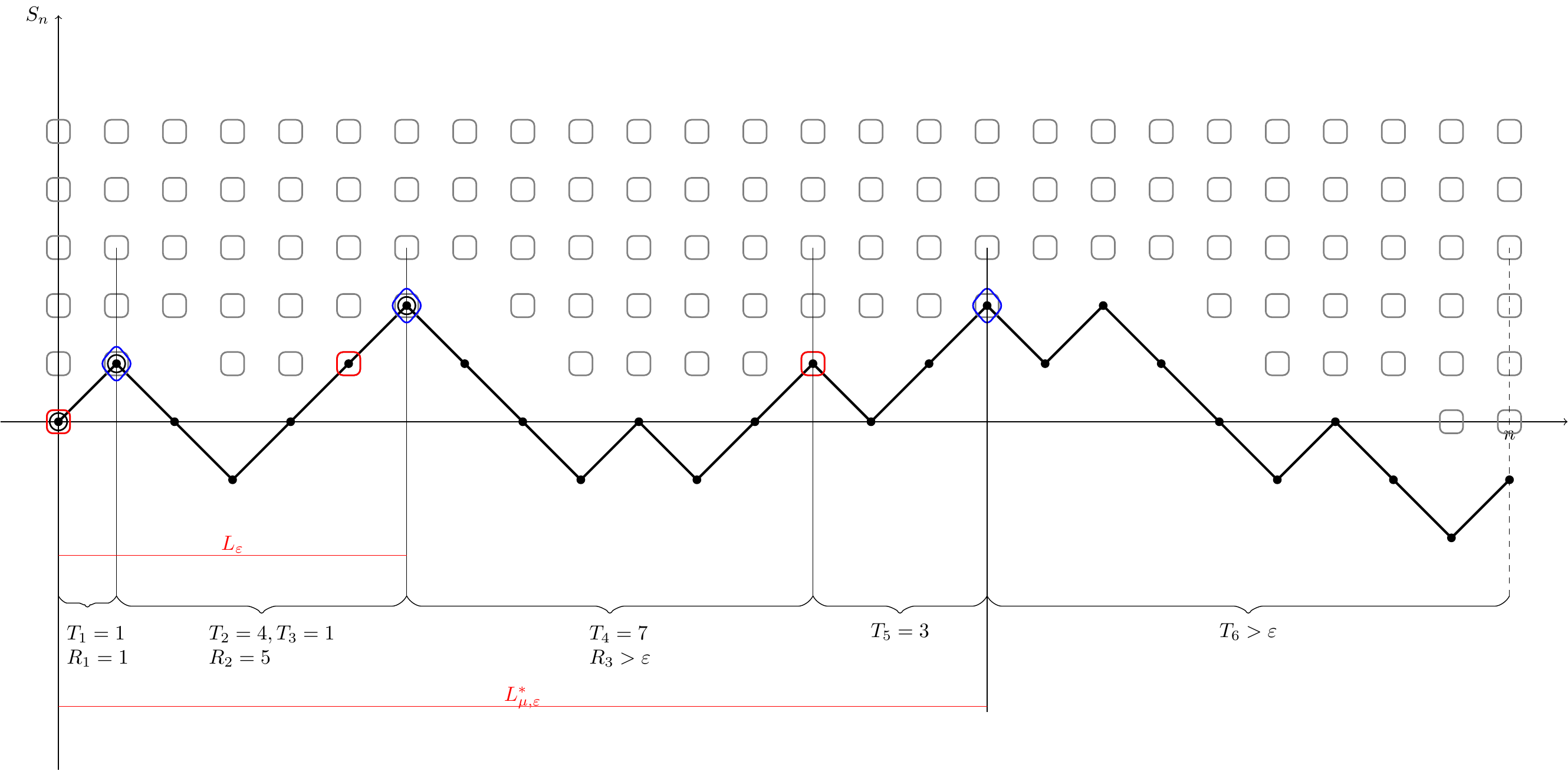}\newline%
\caption{Full versus simplified avalanche length example with flash-crash
trade}%
\label{figfull}%
\end{figure}Formal definitions will be given below.

\section{Simplified avalanche length}
\label{sec:simplified}
\subsection{Discrete results, generating function}
\label{sec:generating}
The simplified avalanche length $L_{\varepsilon}$ is defined as follows: if
 $R_{1}\leq\varepsilon$,\ldots, $R_{k}\leq\varepsilon$, $R_{k+1}%
>\varepsilon$, $k\geq1$, then
\begin{equation}
\label{DefL}L_{\varepsilon}=R_{1}+\ldots+R_{k}.
\end{equation}
This is illustrated in Figure~\ref{figsimple} above particularly for $k=3$ and $\varepsilon=7$.

\begin{proposition}\label{prop1}
The probability generating function for the simplified avalanche length is
given by
\begin{equation}
\label{gfl}E[z^{L_{\varepsilon}}]=\frac{P[R_{1}>\varepsilon]}{E[1-z^{R_{1}}%
;R_{1}\leq\varepsilon]+P[R_{1}>\varepsilon]},
\end{equation}

\end{proposition}

\begin{proof}
From (\ref{DefL}) we have:
\begin{equation}
\begin{split}
E[z^{L_{\varepsilon}}]&=\sum_{k\geq 0} E[z^{R_1+\cdots+R_k}:R_1\leq{\varepsilon}, R_2\leq{\varepsilon},..., R_k\leq{\varepsilon}, R_{k+1}>{\varepsilon}]\\
&= \sum_{k\geq 0} E[z^{R_1}z^{R_2}\cdots z^{R_k}:R_1\leq{\varepsilon}, R_2\leq{\varepsilon},..., R_k\leq{\varepsilon}, R_{k+1}>{\varepsilon}]
\end{split}
\end{equation}
Since $\rho_j$ are the strict ascending ladder times, by following \cite[XIII.1d, P.305]{Fel1} 
it is clear that $R_j$ are  independent and identically distributed (iid). Thus, we obtain 
\begin{equation}
\begin{split}
E[z^{L_{\varepsilon}}]&=\sum_{k\geq 0} E[z^{R_1}:R_1\leq{\varepsilon}]^kE[1:R_{k+1}>{\varepsilon}]=\frac{P[R_{k+1}>{\varepsilon}]}{1-E[z^{R_1}:R_1\leq{\varepsilon}]}\\
&=\frac{P[R_{k+1}>{\varepsilon}]}{E[1:R_1\leq {\varepsilon}]+E[1:R_1>{\varepsilon}]-E[z^{R_1}:R_1<{\varepsilon}]}\\
&=\frac{P[R_{k+1}>{\varepsilon}]}{E[1-z^{R_1}:R_1<{\varepsilon}]+P[R_1>{\varepsilon}]}=\frac{P[R_{1}>{\varepsilon}]}{E[1-z^{R_1}:R_1<{\varepsilon}]+P[R_1>{\varepsilon}]}
\end{split}
\end{equation}
\end{proof}

\begin{remark}
The random variable $R_{1}$ is the first passage time of the random walk
through~$1$. Its distribution can be computed by the reflection principle, and
can be found in \cite[Theorem 2 of III,7, P.89]{Fel1}. Feller defines the probability  that the first passage through $r$ occurs at $n$ by $\varphi_{r,n}$, and by reflection principle:
\begin{equation}
\varphi_{r,n}=\frac{r}{n}\binom{n}{\frac{n+r}{2}}2^{-n}.
\end{equation}
Using this notation we have $P[R_{1}=2n+1]=\varphi_{1,2n+1}$.

 The
corresponding probability generating function can be found in
\cite[XIII,(4.10), P.315]{Fel1}, it is
\begin{equation}
\label{eq:pdfR}
E[z^{R_1}]=\frac{1-\sqrt{1-z^{2}}}{z}.
\end{equation}

\end{remark}

\begin{corollary}
Denoting by $a_k$  the number of the paths of length $k$ representing the first simplified Type~I paths, we obtain:
\begin{equation}
a_k=\frac1{k}\binom{k}{(k+1)/2}
\end{equation}
when $k$ is odd, and $a_k=0$ otherwise.	
\end{corollary}

Since $(R_j)_{j\in\{{1,...,n+1}\}}$ are  independent and identically distributed (iid), instead of $R_1$ from now on we write $R.$
The probability generating function in (\ref{gfl}) can be rewritten in various
different and more explicit ways. We have

\begin{equation}
P[R>\varepsilon]=\frac{3+\varepsilon'}{2+\varepsilon'}\binom{2+\varepsilon'}{\frac{3+\varepsilon'}{2}}/2^{2+\varepsilon'},
\end{equation}

where we write $$\varepsilon'=2\left\lfloor{\dfrac{\varepsilon -1}{2}}\right\rfloor +1,$$
for an easy notation to distinguish odd and even $\varepsilon.$

We have also
\begin{equation}
E[z^{L_{\varepsilon}}]=\frac{1-\Phi_{\varepsilon}(1)}{1-\Phi_{\varepsilon}(z)},
\end{equation}
where $\Phi_{\varepsilon}(z)$ is a polynomial, namely
\begin{equation}
\Phi_{\varepsilon}(z)=\sum_{k=0}^{(\varepsilon'-1)/2}\varphi_
{1,2k+1}z^{2k+1}.
\end{equation}
It can be expressed in terms of the Gaussian hypergeometric function
$_{2}F_{1}(a,b,c,z)$,

\begin{equation}
\Phi_{\varepsilon}(z)=\frac{1-\sqrt{1-z^2}}{z}-\binom{2+\varepsilon'}{\frac{3+\varepsilon'}{2}}\frac{z^{2+\varepsilon'}}{(2+\varepsilon') 2^{2+\varepsilon'}}  {_{2}F}_{1}(1,1+\frac{\varepsilon'}{2},\frac{5+\varepsilon'}{2},z^2).
\end{equation}
\begin{remark}
The hypergeometric function with parameter 1 can be expressed also in terms of (generalizations) of the associated Legendre Polynomials,
see \cite[15.4.13, P.562]{AS}.
\end{remark}

\begin{corollary}
\label{SimplAvalanDist}
\begin{equation}
E[L_{\varepsilon}]=2+\varepsilon'-\frac{2+\varepsilon'}{3+\varepsilon'}\cdot2^{2+\varepsilon'}\bigg/\binom{2+\varepsilon'}{\frac{3+\varepsilon'}{2}},
\end{equation}
\begin{equation}
\begin{split}
 \Var[L_{\varepsilon}]&=\frac{4}{3}\cdot(2+3\varepsilon'+\varepsilon'^2)-\frac{6+7\varepsilon'+2\varepsilon'^2}{3+\varepsilon'}\cdot2^{2+\varepsilon'}\bigg/\binom{2+\varepsilon'}{3+\varepsilon'}\\
&+\frac{(2+\varepsilon')^2}{(3+\varepsilon')^2}\cdot2^{4+2\varepsilon'}\bigg/\binom{2+\varepsilon'}{\frac{3+\varepsilon'}{2}}^2
\end{split}
\end{equation}

\end{corollary}
\begin{proof}
We obtain the moments from the derivatives of the probability generating function at $z=1$.
\end{proof}
\subsection{Brownian excursion limit for the simplified avalanche length}
\label{subsec:simplim}
For a first limit result we consider $\frac1nR$ for $n\geq1$ and introduce the
Laplace-Stieltjes transform for its distribution, namely $E[e^{-\frac{\lambda}{n}R}]$.

\begin{lemma}
We have for fixed $\Re(\lambda)>0$ the asymptotics
\begin{equation}
\label{eq:lapR}
E[e^{-\frac{\lambda}{n}R}]=1-\sqrt{2\lambda}\cdot n^{-\frac12}+\mathcal{O}%
\left(  n^{-1}\right)  ,\quad n\to\infty.
\end{equation}
\end{lemma}
\begin{proof}
This lemma is proved by combining (\ref{eq:pdfR}) with the elementary asymptotics,  namely the exponential series $e^{-\lambda/n}$ as $n\to\infty$, and the power series of $\sqrt{1-z^2}$ as $z\to0$.
\end{proof}
Thus $\frac{1}{n}R$ converges in distribution under $P$ to zero.
This is not surprising, as it is essentially the length of the first excursion
of the random walk. For the Brownian limit this degenerates to zero, as
arbitrary small excursions accumulate near time zero. But the second term in
the asymptotic expansion is relevant for the simplified avalanche length. We
can get another point of view by connecting the limit to Ito's excursion
measure.

Let $G$ denote the distribution of Brownian excursion length under 
the upper\footnote{The distribution under the lower Ito measure $n_{-}$ is the same.}
Ito measure $n_{+}$, see Appendix~\ref{sec:excursion}. Then it is well-known, see \cite[XII]{RY}, that $G$ has a
density, namely
\begin{equation}
g(x)=\frac{1}{\sqrt{2\pi}}x^{-\frac{3}{2}},\quad x>0.
\end{equation}
Note that
\begin{equation}
\int_{0}^{\infty}(1-e^{-\lambda x})g(x)dx=\sqrt{2\lambda}, \quad \Re(\lambda)>0.
\end{equation}
\begin{proposition}
For $n\geq1$ let $\mu_{n}=\Law(\frac{1}{n}R)$.
Then we have
\begin{equation}
\lim_{n\to\infty} \int_{0}^{\infty}(1-e^{-\lambda x})n^{\frac12}\mu
_{n}(dx)=\sqrt{2\lambda} , \quad \Re(\lambda)>0,%
\end{equation}
and thus
$\lim_{n\to\infty}n^{\frac12}\mu_{n}=G$ vaguely on $(0,\infty).$
\end{proposition}
\begin{proof}
This follows from (\ref{eq:lapR}). 
Let $\mu_{n}=\Law(\frac{1}{n}R)$ and $\nu$ denote the measure with density $g$.
Then we can apply Lemma~\ref{lem:vaglim} and 
furthermore  Proposition~(\ref{prop:vaglim}). 
\end{proof}

In the following limit theorem $\varepsilon$ denotes a strictly positive real number. We write for brevity $L_{n\varepsilon}$ instead of the more precise  $L_{\lfloor{n\varepsilon}\rfloor}$. Let \emph{erf} be the \emph{error} function (see \cite{OLBC}), which is defined as:

\begin{equation}
erf(x)=\frac{2}{\sqrt{\pi}} \int_{0}^{x}e^{-t^2}dt.
\end{equation}
\begin{theorem}
The scaled simplified avalanche length for the simple symmetric random
walk converges in distribution to the simplified Brownian avalanche length.
Analytically we have
\begin{equation}
\lim_{n\to\infty} E[e^{-\frac\lambda{n}L_{n\varepsilon}}]= \frac1{\sqrt
{\lambda\varepsilon\pi}\erf(\sqrt{\lambda\varepsilon})+e^{-\lambda\varepsilon}},
\end{equation}
and
\begin{equation}\label{SimpAvalCorr}
\lim_{n\to\infty}E[\frac1{n}L_{n\varepsilon}]=\varepsilon,\quad \lim_{n\to\infty}{\Var}[\frac1{n}L_{n\varepsilon}]=\frac{4}{3}\varepsilon^2.
\end{equation}
\end{theorem}
\begin{proof}

The Laplace
transform of the avalanche length $L_{\varepsilon}$ in the context of Parisian
options is presented in  \cite{DW}. Also, that formula can be derived from the
L\'{e}vy measure of the subordinator consisting of Brownian passage times, see  (Dudok de Wit \cite{DDW}, Theorem~16) .

In \cite{DDW} avalanche length $L$ was defined as a random variable modeling the first time after which no orders get executed  in an $\varepsilon$-time
interval, where the orders are executed due to the price increase. Therefore,  it corresponds to the simplified avalanche length in our model.  
Theorem~16 in \cite{DDW} states that the Laplace transform of the avalanche length is given by
\begin{equation}
E\left[  e^{-\lambda L}\right]  =\frac{1}{\sqrt{\lambda
\varepsilon\pi}\erf\left(  \sqrt{\lambda\varepsilon}\right)  +e^{-\lambda
\varepsilon}}. \label{L-trafo of simple avalanche length}%
\end{equation}

In order to derive the formula (\ref{L-trafo of simple avalanche length}) in \cite{DDW} the author first considered the random variable $L_y$ given by conditioning $L$ on $H = y$ for some $y \geq 0$, where $H$ is the avalanche height defined as the highest level on which the order get executes throughout the period of avalanche. 

Theorem~11 in \cite{DDW} states that the avalanche height $H$ is exponentially distributed with parameter $\sqrt{\frac{2}{\pi \varepsilon}}$, so we have:
\begin{equation}\label{HLaplace}
E\left[  e^{-\lambda L}\right]= E\left[E\left[  e^{-\lambda L}|H\right] \right]=\int_0^{\infty}E\left[  e^{-\lambda L_y}\right] \sqrt{\frac{2}{\pi \varepsilon}} e^{-y\sqrt{\frac{2}{\pi \varepsilon}}}dy.
\end{equation}
Theorem~14 in \cite{DDW} established the Laplace transform of the  $L_y$, and it is
\begin{equation}
E\left[  e^{-\lambda L_{y}}\right]  = \exp \Big(-y \sqrt{\frac{2}{\pi \varepsilon}} (\sqrt{\lambda
\varepsilon\pi}\erf\left(  \sqrt{\lambda\varepsilon}\right)  +e^{-\lambda
\varepsilon}-1)\Big). \label{LyLaplaceTransform}%
\end{equation}

Therefore, the Laplace transform of the avalanche length  (\ref{L-trafo of simple avalanche length}) is obtained by combining (\ref{HLaplace}) and (\ref{LyLaplaceTransform}), i.e.:

\begin{equation}
\begin{split}
E\left[  e^{-\lambda L}\right]&=\int_0^{\infty}\exp \Big(-y \sqrt{\frac{2}{\pi \varepsilon}} (\sqrt{\lambda
\varepsilon\pi}\erf\left(  \sqrt{\lambda\varepsilon}\right)  +e^{-\lambda
\varepsilon}-1)\Big) \sqrt{\frac{2}{\pi \varepsilon}} e^{-y\sqrt{\frac{2}{\pi \varepsilon}}}dy\\
&=\sqrt{\frac{2}{\pi \varepsilon}}  \int_0^{\infty}\exp \Big(-y \sqrt{\frac{2}{\pi \varepsilon}} (\sqrt{\lambda
\varepsilon\pi}\erf\left(  \sqrt{\lambda\varepsilon}\right)  +e^{-\lambda
\varepsilon})\Big) dy\\
&=\frac{1}{\sqrt{\lambda
\varepsilon\pi}\erf\left(  \sqrt{\lambda\varepsilon}\right)  +e^{-\lambda
\varepsilon}}.
\end{split}
\end{equation}

\noindent In the following, the approach is to derive this formula by excursion theory. Note that $G\left(  x\right)$ is defined as
\begin{equation}
G\left(  x\right)  =\int_{x}^{\infty}g\left(  y\right)  \,dy, \quad \text{where}\quad g\left(  y\right)  =\frac{1}{\sqrt{2\pi}}y^{-3/2}.
\end{equation}

\noindent We fix $n\in\mathbb{N}$ such that $n>1/\varepsilon$ and
consider only excursions with length $R_i\in\left[  \frac{1}{n},\varepsilon
\right]  $. Note that there exist only finitely many of those, i.e.
$R_{1}^{\left(  n\right)  },...,R_{j}^{\left(  n\right),  }$  and then there is  a
first excursion $R_{j+1}$   (this does not depend on $n$), whose length is more than $\varepsilon$. Define the modified avalanche length $L_{\varepsilon,n}$ by
\begin{equation}
L_{\varepsilon,n}=R_{1}^{\left(  n\right)  }+...+R_{j}^{\left(  n\right)
}\leq L_{\varepsilon}%
\end{equation}
where $L_{\varepsilon}$ is the combined length of all excursions until one
whose length exceeds $\varepsilon$, see (\ref{DefL}).
Since there are a $\sigma$-discrete
number of excursions we translate into the continuous time case, following the geometric series. 
 Since the $R_{i}^{(n)}$ are iid, we sum the
geometric series:
\begin{equation}
E\left[  e^{-\lambda L_{\varepsilon,n}}\right]  =\frac{P\left(  R>\varepsilon
\right)  }{1-E\left[  e^{-\lambda R^{\left(  n\right)  }}\mathbb{I}%
_{R^{(n)}\leq\varepsilon}\right]  }.
\end{equation}

\noindent Letting $n\rightarrow\infty$, we obtain by dominated convergence
that%
\begin{equation}
E\left[  e^{-\lambda L_{\varepsilon}}\right]  =\frac{P\left(  R>\varepsilon
\right)  }{1-E\left[  e^{-\lambda R}\mathbb{I}_{R\leq\varepsilon}\right]  }.
\end{equation}

\noindent In order to estimate the denominator
\begin{equation}
E\left[  \mathbb{I}_{R\leq\varepsilon
}\right] + E\left[  \mathbb{I}_{R>\varepsilon
}\right] -E\left[  e^{-\lambda R}\mathbb{I}_{R\leq\varepsilon}\right] =E\left[  \left(  1-e^{-\lambda R}\right)  \mathbb{I}_{R\leq\varepsilon
}\right]  +P\left(  R>\varepsilon\right),
\end{equation}

\noindent we define
\begin{align}
\phi\left(  x\right)     =\left(  1-e^{-\lambda x}\right)  \mathbb{I}%
_{x\leq\varepsilon}.
\end{align}

\noindent Thus:
\begin{align}
\phi^{\prime}\left(  x\right)     =\lambda e^{-\lambda x}\mathbb{I}%
_{x\leq\varepsilon}+\left(  e^{-\lambda x}-1\right)  \delta_{\varepsilon
}\left(  x\right).
\end{align}
Following the distribution of $R$ under the Ito
measure $\,n$ the first equality in (\ref{dokaz}) is obtained by the Master Formula (\ref{MasterFormula}), see \cite[XII, Proposition1.10]{RY}.  Further, we progress by Fubini theorem, so we obtain:
\begin{equation}\label{dokaz}
\begin{split}
E\left[  \left(  1-e^{-\lambda R}\right)  \mathbb{I}_{R\leq\varepsilon
}\right] & +E\left[  \mathbb{I}_{R>\varepsilon
}\right]  =\int\phi\left(  R\left(
u\right)  \right)  \,n\left(  du\right)  +\int\mathbb{I}_{R(u)>\varepsilon
}\,n\left(  du\right) \\
&  =\int_{0}^{\infty}\phi^{\prime}\left(  x\right)  \,G\left(  x\right)
\,dx+G\left(  \varepsilon\right) \\
&  =\lambda\int_{0}^{\varepsilon}e^{-\lambda x}G(x)\,dx+e^{-\lambda
\varepsilon}G\left(  \varepsilon\right) -G\left(  \varepsilon\right) + G\left(  \varepsilon\right) \\
&  =\left(  e^{-\lambda\varepsilon}-1\right)  G\left(  \varepsilon\right)
+\lambda\int_{0}^{\varepsilon}e^{-\lambda x}G(x)\,dx+G\left(  \varepsilon
\right) \\
&  =\left(  e^{-\lambda\varepsilon}-1\right)  G\left(  \varepsilon\right)
+\lambda\int_{0}^{\varepsilon}e^{-\lambda x}G(x)\,dx+\int
_{\varepsilon}^{\infty}g(x)\,dx \\
&  =\int_{0}^{\varepsilon}\left(  1-e^{-\lambda x}g(x)\right)  \,dx+\int
_{\varepsilon}^{\infty}g(x)\,dx,
\end{split}
\end{equation}

where last equality is obtained by the partial integration.

Furthermore, (\ref{SimpAvalCorr}) is obtained from differentiating the generating function.
\end{proof}

We note that the moments from  (\ref{SimpAvalCorr}) are consistent with the asymptotics of the moments from Corollary \ref{SimplAvalanDist}.

\begin{remark}
The distribution of the simplified Brownian avalanche length was obtained by
\cite{DDW}.
We see that the limit of the random walk
avalanche length agrees with the Brownian avalanche length. To prove that the
limit distribution actually \emph{is} the Brownian avalanche length
distribution, without referring to the result in \cite{DDW}, would require
further justification, perhaps a continuity argument for the avalanche length
as a functional on the Skorohod (or Wiener) space. This is postponed for now.
\end{remark}

The analytical result above allows to study the distribution of $L_{\varepsilon}$ in more detail. We conjecture it has an exponentially decaying tail, and thus positive integer moments of all orders, which can be obtained from differentiating the probability generating function \cite{HWP}.

\section{More excursions: The full avalanche length}
\label{sec:more}
In this section we are going to analyze the full avalanche length with parameter $\mu \geq 1$ and windows parameter $\varepsilon>0$, which we denote
by ${L^{*}_{\mu,\varepsilon}}$.
The order execution avalanche continues, until there is a time interval
of length greater then $\varepsilon$ that contains neither Type~I nor Type~II trades,
and thus the full avalanche length is obtained from a sequence of trades with
intratrading time less than $\varepsilon$ followed by a trade with intratrading
time bigger than~$\varepsilon$.

Formally ${L^{*}_{\mu,\varepsilon}}$ is defined as follows: If
$k\geq1$ and $T_{1}\leq\varepsilon$,\ldots, $T_{k}\leq\varepsilon$, $T_{k+1}>\varepsilon$, then
\begin{equation}
{L^{*}_{\mu,\varepsilon}}=T_{1}+\ldots+T_{k}.
\end{equation}
\subsection{Generating functions for the full avalanche length}
\label{subsec:genfull}
Given the parameter $\mu\geq1$ we define three sets of random walk paths of length $n\geq1$ by

\begin{equation}
\mathscr{A}_{n,\mu}=\left\{s\in U_{n}:\mbox{$s_0=0$,
$-\mu<s_k\leq0$ for $0< k\leq n-1$ and $s_n=+1$}\right\},	`
\end{equation}
\begin{equation}
\mathscr{B}_{n,\mu}=\left\{s\in U_{n}:\mbox{$s_0=0$,
$-\mu<s_k\leq0$ for $0< k< n-1$, $s_{n-1}=0$ and $s_n=-1$}\right\},
\end{equation}

\begin{equation}
\mathscr{C}_{n,\mu}=\left\{s\in \mathcal{A}_{n,\mu}:\mbox{
$\min(s_0,\ldots,s_{n-1})=-\mu+1$}\right\}  .
\end{equation}
Next we define the corresponding sets with arbitrary, finite path length, that is,
\begin{equation}
\mathscr{A}_{\mu}=\bigcup_{n\geq1}\mathscr{A}_{n,\mu},\quad\mathscr{B}_{\mu
}=\bigcup_{n\geq1}\mathscr{B}_{n,\mu},\quad\mathscr{C}_{\mu}=\bigcup_{n\geq
1}\mathscr{C}_{n,\mu}.
\end{equation}
Examples are given in Figure~\ref{figa}, Figure \ref{figb} and Figure \ref{figc}.

        \begin{figure}[ht!]
            \includegraphics[width=.13\textwidth]{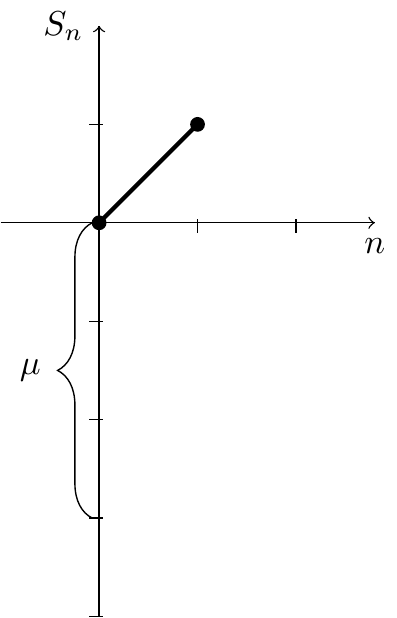}\hfill
            \includegraphics[width=.2\textwidth]{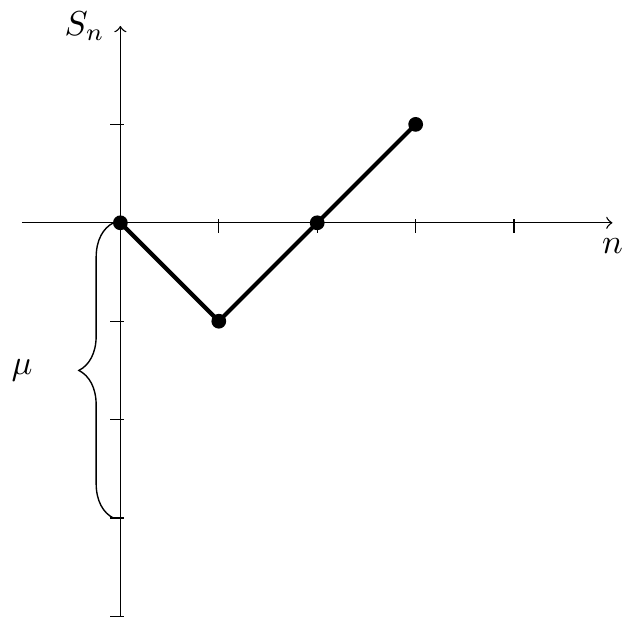}\hfill
            \includegraphics[width=.2\textwidth]{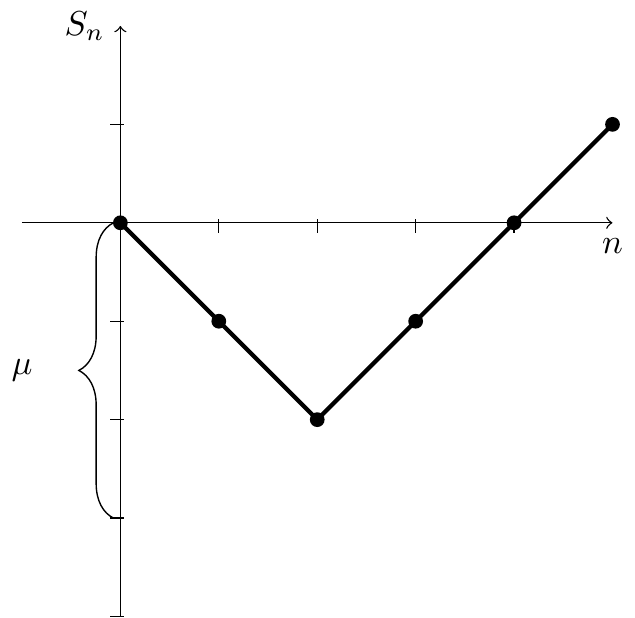}\hfill
            \includegraphics[width=.27\textwidth]{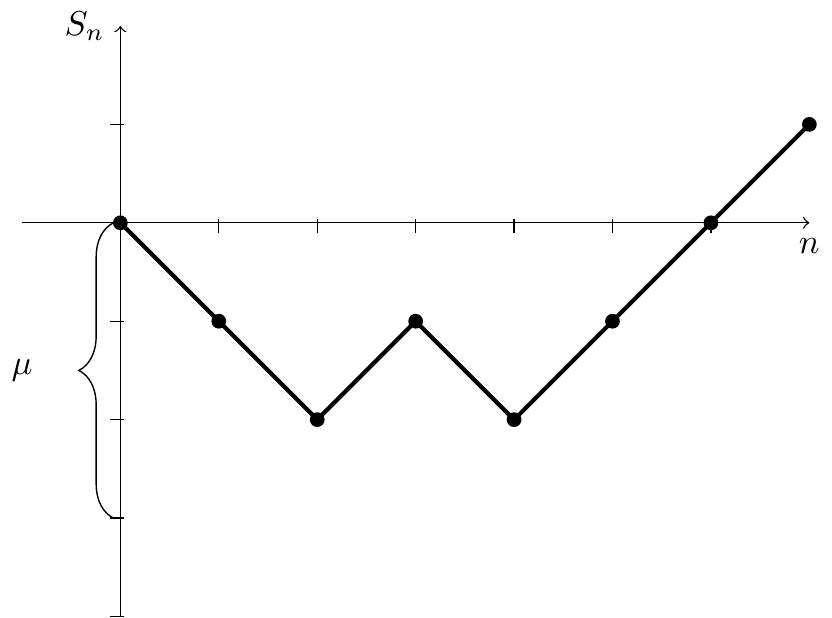}
            \caption{Some elements of $\mathscr{A}_{\mu}$}
            \label{figa}
        \end{figure}

 \begin{figure}[ht!]
            \includegraphics[width=.13\textwidth]{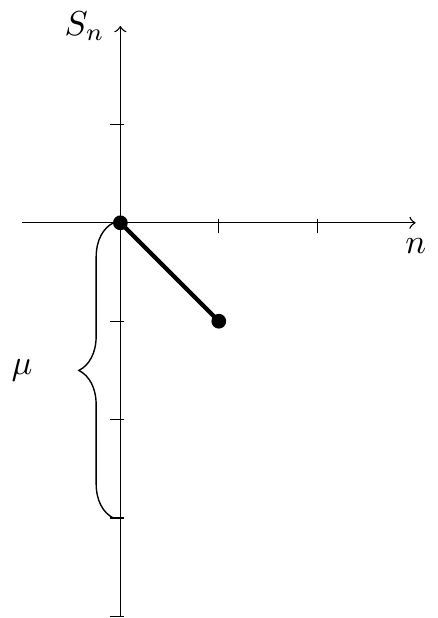}\hfill
            \includegraphics[width=.2\textwidth]{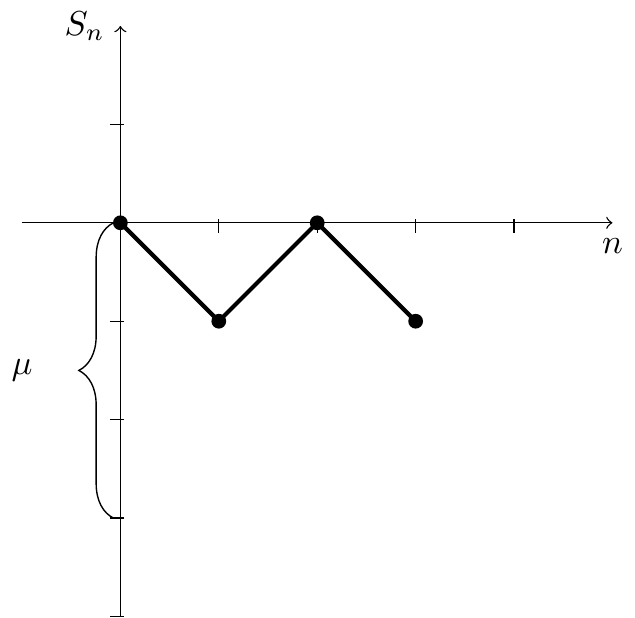}\hfill
            \includegraphics[width=.2\textwidth]{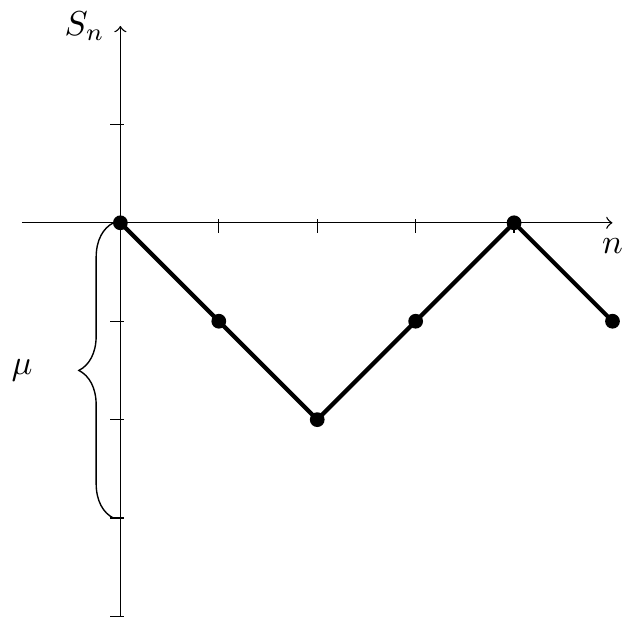}\hfill
            \includegraphics[width=.27\textwidth]{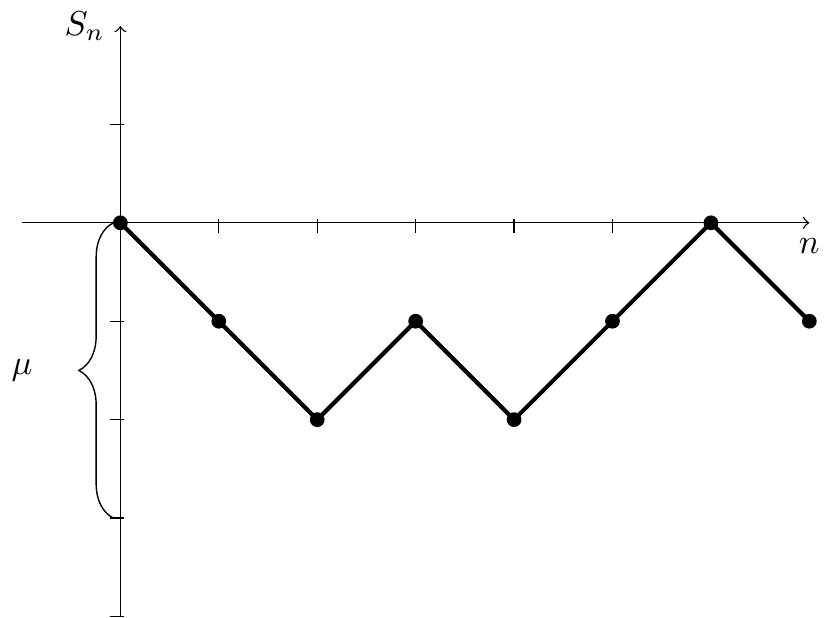}
            \caption{Some elements of $\mathscr{B}_{\mu}$}
            \label{figb}
        \end{figure}

 \begin{figure}[ht!]
            \includegraphics[width=.17\textwidth]{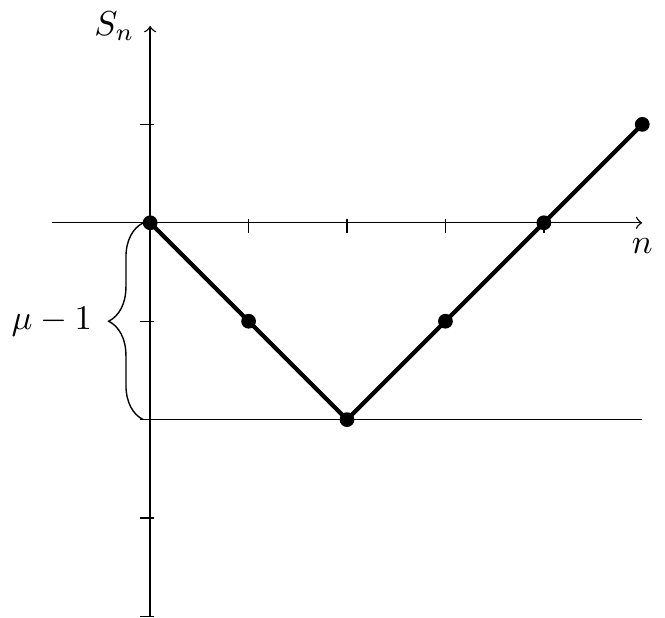}\hfill
            \includegraphics[width=.2\textwidth]{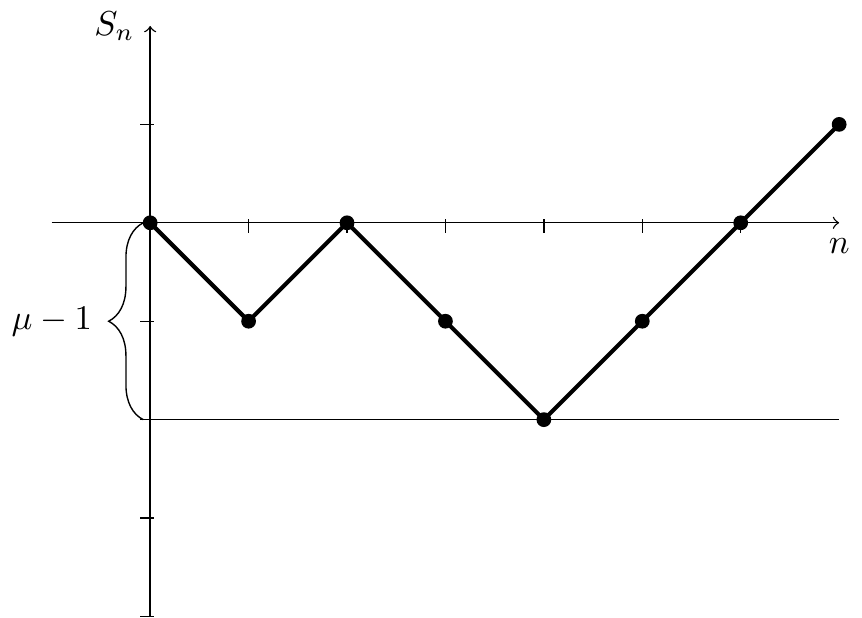}\hfill
            \includegraphics[width=.2\textwidth]{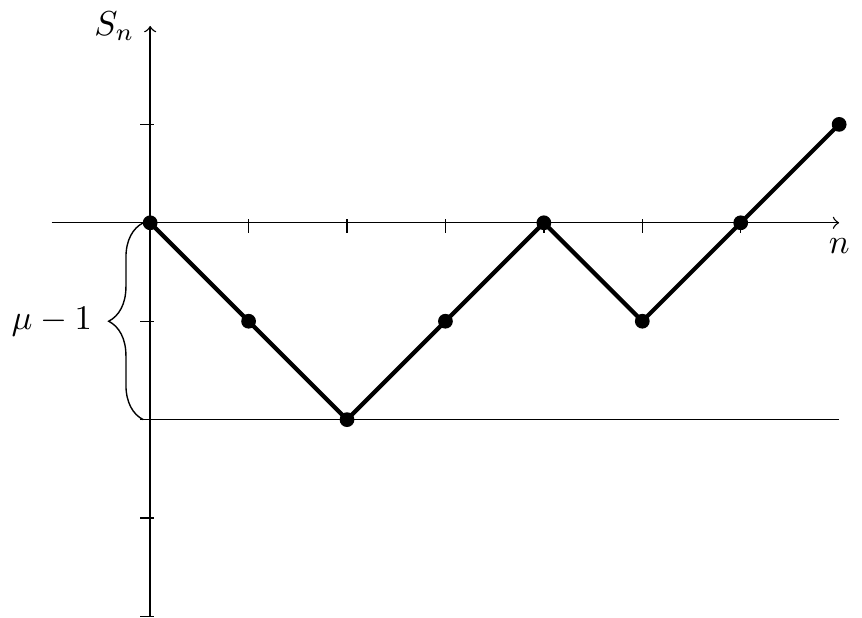}\hfill
            \includegraphics[width=.23\textwidth]{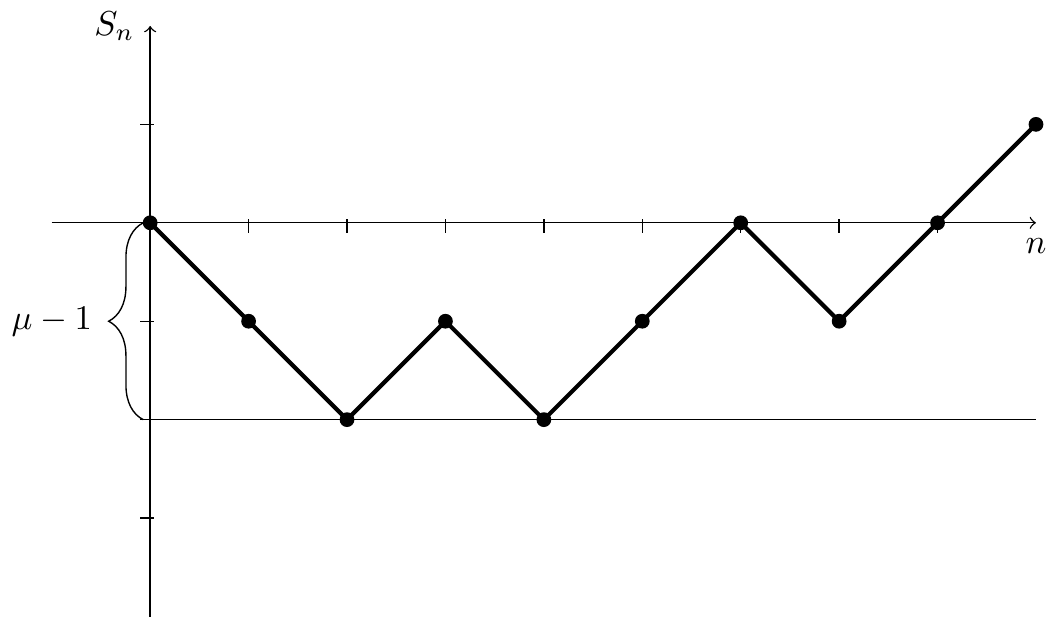}
            \caption{Some elements of $\mathscr{C}_{\mu}$}
            \label{figc}
        \end{figure}

We will count all trading paths of length $n$ and of a certain type by combinatorial 
enumeration. If $n\geq0$ and $(s_{0},\ldots ,s_{n})$ is a path we assign the \emph{size}~$n$ 
to it. We associate with Cartesian products of sets of trading paths the set of all
\emph{concatenations} of the paths from the components, and obviously their
sizes are added. As we have the symmetric Bernoulli model, the corresponding
probability generating functions are obtained from the ordinary generating
functions by the substitution $z\mapsto z/2$.

As in \cite[XIV.4]{Fel1} let us introduce
\begin{equation}
\lambda_{1}(s)=\frac{1+\sqrt{1-s^{2}}}{s},
\quad
\lambda_{2}(s)=\frac{1-\sqrt{1-s^{2}}}{s}.
\end{equation}
Then we have the following lemma:
\begin{lemma}\label{GFABC}
The probability generating functions for the combinatorial classes
$\mathscr{A}_{\mu},\mathscr{B}_{\mu},\mathscr{C}_{\mu}$ are given by
\begin{eqnarray}
A_{\mu}(s)&=&\frac{\lambda_{1}(s)^{\mu}-\lambda_{2}(s)^{\mu}}{\lambda
_{1}(s)^{\mu+1}-\lambda_{2}(s)^{\mu+1}},\label{gfa}\\
B_{\mu}(s)&=&A_{\mu}(s),\label{gfb}\\
C_{\mu}(s)&=&A_{\mu}(s)-A_{\mu-1}(s). \label{gfc}%
\end{eqnarray}
\end{lemma}
\begin{proof}
Shifting a path from $\mathscr{A}_\mu$ by $\mu$ we obtain a path corresponding to absorption at zero
at the $n$-th trial in a symmetric game of gamblers ruin with initial position $\mu$ and absorbing barriers
at $0$ and $\mu+1$. This generating function is derived in \cite[XIV.4, (4.12), P.351]{Fel1}. Using that
formula for $p=1/2$, $q=1/2$, $z=\mu$, $a=\mu+1$ yields (\ref{gfa}).
The last step for a path in $\mathscr{B}_{n,\mu}$ is down. If we replace it by an up-step,
we obtain a path in $\mathscr{A}_{n,\mu}$. Thus we have a bijection from $\mathscr{B}_{n,\mu}$ to
$\mathscr{A}_{n,\mu}$ yielding (\ref{gfb}).
The set $\mathscr{C}_{n,\mu}$ contains pathes of length $n$, that hit $-\mu+1$, but not $-\mu$, thus
\begin{equation}
\mathscr{C}_{n,\mu}=\mathcal{A}_{n,\mu}\setminus\mathscr{A}_{n,\mu-1}.
\end{equation}
Since $\mathscr{A}_{n,\mu-1}\subseteq \mathscr{A}_{n,\mu}$ equation (\ref{gfc}) follows.
\end{proof}\

Next we state the key result for the description of Type~II trades, with is a path decomposition
for trading excursions that lead to a Type~II trade.
which is, in fact, related to the decomposition illustrated in \cite[Sec.VI.3, Fig.5,P.256]{RY}. We state first the path decomposition and the proof, which refers to a few lemmas, which we provide for the clarity of exposition afterwards, at the end of this subsection.

\begin{proposition}
\label{prop:decomp} 
Suppose that the first trade is a Type~II. Then there exists an integer $K\geq1$,
such that $(S(0),\ldots,S(T_1))$ is the concatenation of $K$ elements from $\mathscr{B}_\mu$ and
 one element from $\mathscr{C}_\mu$.
\end{proposition}

\begin{figure}[h]

\includegraphics[scale=.55]{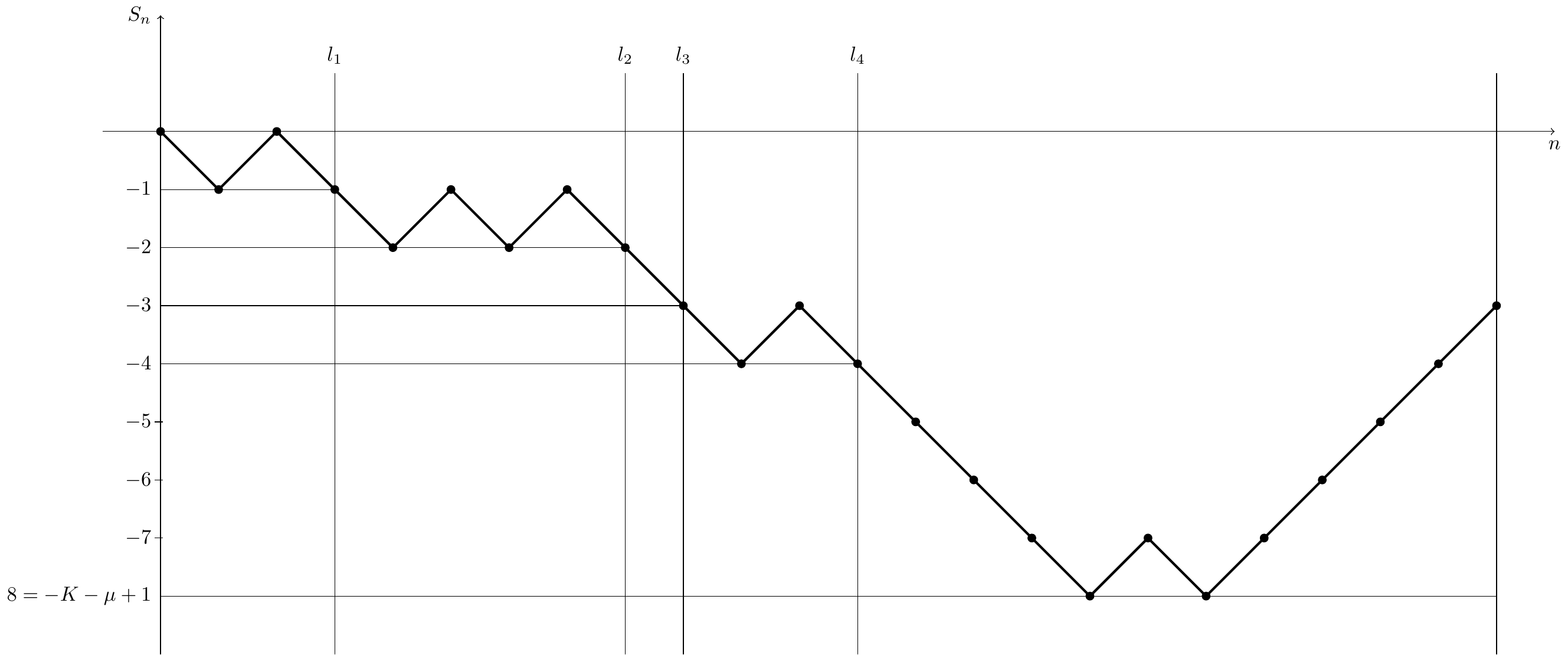}
\caption{Path decomposition of a Type~II trading excursion, $\mu=5$, $K=4$}
\label{fig:deco}
\end{figure}

\begin{proof}
Assume that the first trade is Type II. Then $(S(0),...,S(T_1))$ contains no Type I trade, thus $S(0),...,S(T_1)$ are non positive by Lemma \ref{LemA}.
A Type~II trade occurs when the price goes down by $\mu$ or more steps and then goes up
by $\mu$ steps. The first time this behavior is completed, the excursion ends.

Let $K=-S(T_1-1)$ and $\ell^*=\max\{0\leq n<T_1-1:S(n)=-K\}$, see Figure \ref{fig:deco} for an illustration particularly for  $\mu=5$, $K=4$ and $\ell^*=\ell_4$.  Since the Type II trade is completed, the path $(S(\ell^*),...,S(T_1))$ must go down to the level $-K-\mu+1$, and at that moment the new order is placed at the level $-K+1$. Then, the Type II trade occurs at time $T_1$, i.e. $S(T_1)=-K+1$. Note that if the path $(S(\ell^*),...,S(T_1))$ did go deeper, a Type~II trade would be followed by Type~I trade which would complete before $T_1$. Thus this part of the path, namely  the path $(S(\ell^*),...,S(T_1))$, belongs to the class $\mathscr{C}_\mu$.

Further, note that  every level from $0$ to $-K$ is visited by $(S(0),...,S(\ell^*))$. 
Let $\ell_0=0$ and for $k=1,...,K$ define $\ell_k$ as the \emph{last} time level $-k$ is visited before $\ell^*$ \emph{plus one}, i.e. $\ell^k=\max\{0 < n<\ell^*+1:S(n)=-k\}.$ Since there is no trade at all, in particular no Type~I trade during this initial period, by Lemma \ref{EarlyII}, the paths $(S(\ell_{j-1}),...,S(\ell_{j}))$, $j=1,...,K$, have depth less than $\mu$. Therefore,  there are $K
$ paths, precisely $(S(\ell_{j-1}),...,S(\ell_{j}))$ $j=1,...,K$, that  belong to class $\mathscr{B}_\mu$.
\end{proof}

\begin{remark}
The $\ell_{j}$ are not stoppimg times, but honest times.
\end{remark}

\begin{remark}
The final down segments in the $\mathscr{B}_\mu$ paths seem to vanish in the Brownian limit, but they correspond to the support of the local time at the minimum.
\end{remark}

\begin{lemma}
\label{down} If the first trade is a Type~II trade, then
\begin{equation}
\max\{S(0),\ldots,S(T_{1})\}\leq0, \quad\min\{S(0),\ldots,S(T_{1})\}\leq-\mu.
\end{equation}

\end{lemma}

\begin{proof}
Assume by contradiction there is a time $n\in[0,T_1]$ with $S_n>0$.
We can assume that this $n$ is the minimal number with this property.
Then we had $V(n,S_n)>0$ and $S_n>S_0$, thus the first trade
would be a Type~I trade.
Assume next by contradiction that $-\mu<S_n\leq0$ for all $n\in[0,T_1]$.
Then the recursion for $\alpha_n$ shows $\alpha_n=1$ for all $n\in[1,T_1]$
and no trade could take place.
\end{proof}

\begin{lemma}
\label{EarlyII} Suppose $0\leq a\leq b$ and S(a)=S(b). If
\begin{equation}
\quad\min\{S(a),\ldots,S(b)\} < S(a)-\mu,
\end{equation}
there is a Type~II trade followed by a Type~I trade in $[a,b]$.
Moreover, if 
\begin{equation}
\quad\min\{S(a),\ldots,S(b)\}=S(a)-\mu,
\end{equation}
there is a Type~II trade at time $b$.
\end{lemma}

\begin{proof}
Let
\begin{equation}
\gamma=\min(S(a),\ldots,S(b),\quad
\beta=\min\{n\in[a,b]:S(n)=\gamma\},\quad,
\end{equation}
and
\begin{equation}
\delta=\min\{n\in[\beta,b]:S(n)=\gamma+\mu\}.
\end{equation}
At time $\beta$ a new order is placed at the price level $S_n+\mu=\gamma+\mu$, and therefore we have $V(\beta+1,\gamma+\mu)>0$. This order at price level $\gamma+\mu$ get executed at time $\delta$.

The second part of the lemma follows directly from the definition of the Type~II trade.

\end{proof}

\begin{lemma}
\label{LemA} 
The first trade after zero is a Type~I trade iff
\begin{equation}
\label{MaxI}\max(S(0),\ldots,S(T_{1}-1)\}=0,\quad S(T_{1})=1
\end{equation}
and
\begin{equation}
\label{MinI}\min(S(0),\ldots,S(T_{1}-1))>-\mu.
\end{equation}

\end{lemma}

\begin{proof}
Suppose that the first trade is Type~I trade. Then there is no Type~II trade
during the time interval $[0,T_1]$, and thus by Lemma~\ref{EarlyII} inequality~(\ref{MinI}) holds. Note that the price level 1 is the strict ladder time, and therefore,  according to Lemma~\ref{EasyI}, a Type~I trade occurs  at the price level 1. Since the first trade occurs at time $T_1$,  it directly implies that (\ref{MaxI}) holds. Conversely, assume that (\ref{MinI}) and (\ref{MaxI}) hold. Then, by Lemma~\ref{EasyI} a Type~I trade ocurs at time  $T_1$, and there is no earlier Type~I trade. Further, Lemma~\ref{down} shows there is no Type~II trade earlier.
\end{proof}

\subsection{Generating function for the full avalanche length}
\begin{proposition}
\label{prop:pgfT}
The generating function of the time to the next trade is given by
\begin{equation}
\label{gfT1}
E[z^{T_1}]=A_{\mu}(z)+\frac{B_{\mu}(z)}{1-B_{\mu}(z)}\cdot C_{\mu}(z).
\end{equation}
\end{proposition}
\begin{proof}
Let us define the set $\mathscr{T}_{\mu,n}$ of all paths where the first trade occurs at step $n$.
and $\mathscr{T}_\mu=\bigcup_{n\geq1}\mathscr{T}_{\mu,n}$.
This set contains elements that are either Type~I trades, and thus correspond
by Lemma~\ref{LemA} to the class $\mathscr{A}_\mu$ or Type~II trades, which correspond
by the path decomposition in Proposition~\ref{prop:decomp} bijectively to the Cartesian product
of the set of finite sequences of length $1$ or more of elements in $\mathscr{B}_\mu$ times
the class $\mathscr{C}_\mu$.
Using the symbolic notation from \cite[Section~I]{FS} this is written more clearly as
\begin{equation}
\label{eq:classT}
\mathscr{T}_\mu=\mathscr{A}_\mu\cup
\SEQ_{\geq1}(\mathscr{B}_\mu)\times\mathscr{C}_\mu.
\end{equation}
From \cite[Theorem~I.1, P.27]{FS} and the section on {\em restricted constructions}, in particular
$\SEQ_{\geq k}$, in \cite[P.30]{FS}, we obtain (\ref{gfT1}).
\end{proof}
Once the law of $T_1$ is found, we can compute the
avalanche length distribution.
\begin{theorem}
\label{thm:full}
The full avalanche length for the symmetric random walk has probability
generating function
\begin{equation}\label{eq:full}
E[z^{L^{*}_{\mu\varepsilon}}]=\frac{P[T_1>\varepsilon]}{E[1-z^{T_1};T_1\leq\varepsilon
]+P[T_1>\varepsilon]},
\end{equation}
where $T_1$ has the probability generating function given in (\ref{gfT1}) above.
\end{theorem}
\begin{proof}
This follows exactly from the same arguments as in the proof of Proposition~\ref{prop1} with $R$ replaced by $T$.
\end{proof}
Let us now consider also the following quantities:
The time to the first trade, assuming it is a Type~I trade, that is, $T_1$ on $\{S(T_1)>0\}$;
the time to the first trade, assuming it is a Type~II trade, that is,  $T_1$ on $\{S(T_1)\leq0\}$;
and the time to the first Type~II trade.

Furthermore, denote by $D$ the index of the first Type~II trade in the sequence of all trades, i.e.:
\begin{equation}
D=\inf\{i\geq1:S(\tau_{i})<S(\tau_{i-1})\}.
\end{equation}
Consequently the time to the first Type~II trade is $\tau_D.$
\begin{lemma}%
We have \footnote{We use semicolon notation, $E[X;A]=E[XI_A]$ for an integrable random variable X and an event A.}
\begin{equation}\label{firstType1}
E[z^{T_1};S(T_1)>0]=A_{\mu}(z), \quad
E[z^{T_1};S(T_1)\leq0]=\frac{B_{\mu}(z)}{1-B_{\mu}(z)}\cdot C_{\mu}(z) 
\end{equation}
and
\begin{equation}
E[z^D]=\frac1{1-E[z^{T_1};S(T_1)\leq0]}\cdot E[z^{T_1};S(T_1)\leq0].
\end{equation}
\end{lemma}
\begin{proof}
By following Lemma~\ref{LemA} a Type~I trading excursion corresponds to an element in $\mathscr{A}_\mu$, and therefore by \cite[Theorem~I.1, P.27]{FS} the left equation in (\ref{firstType1}) holds. By Proposition~\ref{prop:decomp} path of a Type~II trading excursion corresponds to the second term on the right hand side of (\ref{eq:classT}),  and thus by \cite[Theorem~I.1, P.27]{FS} the right equation in (\ref{firstType1}) holds. 
The path to the next Type~II trades consists of a (possible empty) sequence of
Type~I trading excursions followed by a Type~II trading excursions, thus
to $\SEQ(\mathscr{B}_\mu)\times\mathscr{C}_\mu$.
We can apply again \cite[Theorem~I.1, P.27]{FS} to obtain the corresponding generating functions.
\end{proof}
\begin{corollary}
\begin{equation}
P[S(T_1)>0]=\frac{\mu}{\mu+1}
\end{equation}
\end{corollary}
\begin{proof}
We set $z=1$ in the previous result.
\end{proof}

\subsection{The particular probabilities}
By combining (\ref{gfT1}) with (\ref{gfa}), (\ref{gfb}) and (\ref{gfc})  we have
\begin{equation}\label{Tmu}
\begin{split}
T_{\mu}(z)=A_{\mu}(z)+\frac{A_{\mu}(z)}{1-A_{\mu}(z)}(A_{\mu}(z)-A_{\mu-1}(z))=\frac{A_{\mu}-A_{\mu}(z)A_{\mu-1}(z)}{1-A_{\mu}(z)}.
\end{split}
\end{equation}
\begin{table}
\begin{center}
\begin{tabular}{|l||*{10}{c|}}
\hline
\backslashbox{$\mu$}{$n$}&\makebox[1.5em]{1}&\makebox[1.5em]{2}&\makebox[1.5em]{3}
&\makebox[1.5em]{4}&\makebox[1.5em]{5}&\makebox[1.5em]{6}&\makebox[1.5em]{7}&\makebox[1.5em]{8}&\makebox[1.5em]{9}&\makebox[1.5em]{10}\\\hline\hline
1 &1/2&1/4&1/8&1/16&1/32&1/64&1/128&1/256&1/512&1/1024\\\hline
2 &1/2&0&1/8&1/16&1/16&3/64&5/128&1/32&13/512&21/1024\\\hline
3&1/2&0&1/8&0&1/16&1/64&5/128&5/256&7/256&19/1024\\\hline
4&1/2&0&1/8&0&1/16&0&5/128&1/256&7/256&7/1024\\\hline
5&1/2&0&1/8&0&1/16&0&5/128&0&7/256&1/1024\\\hline
6&1/2&0&1/8&0&1/16&0&5/128&0&7/256&0\\\hline
7&1/2&0&1/8&0&1/16&0&5/128&0&7/256&0\\\hline
\end{tabular}
\end{center}
\caption{The probabilities $p_n=P[T_1=n]$ for
 $n=1,...,10$ and $\mu=1,...,7$}\label{Tablepn}
\end{table}

For  $\mu=1,2,3,4,5,6,7$, we use series expansion and expand the corresponding formula (\ref{Tmu}). Thus, we obtain probabilities $p_n=P[T_1=n]$,  for particular values see Table~\ref{Tablepn}, when $n=1,...,10$ and $\mu=1,...,7$. Furthermore, the following holds:
\begin{equation}
T_{\mu}(z)=\sum_{n\geq0}(\sum_{k=0}^{n}p_k)z^n(1-z).
\end{equation}

Let  $q_n=P[T_1>n]$, i.e. $q_n=\sum_{k\geq n+1}P[T_1=k]$, and let $Q_{\mu}(z)=\sum_{n\geq0}q_nz^n$. Then, we have $$Q_{\mu}(z)=\sum_{n\geq0}(\sum_{k\geq n+1}p_k)z^n=\frac{1-T_{\mu}(z)}{1-z}.$$
For particular  values of probabilities $q_{\varepsilon}=P[T_1>{\varepsilon}]$ when
 ${\varepsilon}=1,...,10,11$ and $\mu=1,...,7$ see Table~\ref{Tableqn}.

\begin{table}[h]
\begin{center}
\begin{tabular}{|l||*{9}{c|}}
\hline
\backslashbox{$\mu$}{$\varepsilon$}&\makebox[1.5em]{1}&\makebox[1.5em]{2}&\makebox[1.5em]{3}
&\makebox[1.5em]{4}&\makebox[1.5em]{5}&\makebox[1.5em]{6}&\makebox[1.5em]{7}&\makebox[1.5em]{8}&\makebox[1.5em]{9}\\\hline\hline
1 &1/2&1/4&1/8&1/16&1/32&1/64&1/128&1/256&1/512\\\hline
2 &1/2&1/2&3/8&5/16&1/4&13/64&21/128&17/128&55/512\\\hline
3&1/2&1/2&3/8&3/8&5/16&19/64&33/128&61/256&27/128\\\hline
4&1/2&1/2&3/8&3/8&5/16&5/16&35/128&69/256&31/128\\\hline
5&1/2& 1/2& 3/8&3/8& 5/16&5/16&35/128&35/128& 63/256\\\hline
6&1/2&1/2& 3/8&3/8& 5/16& 5/16& 35/128& 35/128& 63/256\\\hline
7&1/2& 1/2&3/8& 3/8& 5/16& 5/16& 35/128& 35/128& 63/256\\\hline
\end{tabular}
\end{center}
\caption{The probabilities $q_{\varepsilon}=P[T_1>{\varepsilon}]$ for
 ${\varepsilon}=1,...,9$ and $\mu=1,...,7$}\label{Tableqn}
\end{table}

Recall (\ref{eq:full}), i.e. we have:
\begin{equation}\label{gfDerive}
E[z^{L^{*}_{\mu\varepsilon}}]=\frac{P[T_{1}>{\varepsilon}]}{1-E[z^{T_1}:T_1\leq{\varepsilon}]}=\frac{q_{\varepsilon}}{1-\sum_{k=1}^{\varepsilon}z^kp_k}.
\end{equation}
Thus, by plugging particular values from Table~\ref{Tablepn} and Table~\ref{Tableqn},  when $\mu=1, 2, 3, 4$ and $\varepsilon=1, 2,3,4,5$, in equation (\ref{gfDerive}),  the probability generating functions of the full avalanche length $L_{\mu,\varepsilon}^*$ are derived, see Table~\ref{TableFullGF}. 

\begin{table}[h]
\begin{center}
\begin{tabular}{|l||*{11}{c|}}
\hline
\backslashbox{$\mu$}{$\varepsilon$}&\makebox[3.5em]{1}&\makebox[3.5em]{2}&\makebox[3.5em]{3}
&\makebox[3.5em]{4}&\makebox[3.5em]{5}\\\hline\hline
1 &$\frac{1}{2-z}$&$\frac{1}{4-2z-z^2}$&$\frac{1}{8-4z-2z^2-z^3}$&$\frac{1}{16-8z-4z^2-2z^3-z^4}$&$\frac{1}{32-16z-8z^2-4z^3-2z^4-z^5}$\\\hline
2 &$\frac{1}{2-z}$&$\frac{1}{2-z}$&$\frac{3}{8-4z-z^3}$&$\frac{5}{16-8z-2z^3-z^4}$&$\frac{4}{16-8z-2z^3-z^4-z^5}$\\\hline
3&$\frac{1}{2-z}$&$\frac{1}{2-z}$&$\frac{3}{8-4z-z^3}$&$\frac{3}{8-4z-z^3}$&$\frac{5}{16-8z-2z^3-z^5}$\\\hline
4&$\frac{1}{2-z}$&$\frac{1}{2-z}$&$\frac{3}{8-4z-z^3}$&$\frac{3}{8-4z-z^3}$&$\frac{5}{16-8z-2z^3-z^5}$\\\hline
\end{tabular}
\end{center}
\caption{The probability generating function of the full avalanche length $L_{\mu,\varepsilon}^*$ when $\mu=1, 2, 3, 4$ and $\varepsilon=1, 2,3,4,5$}\label{TableFullGF}
\end{table}

Furthermore, we expand the probability generating function of  the full avalanche length presented in Table~\ref{TableFullGF} and calculate  the individual probabilities probabilities  see Table~\ref{TableAvamu1} (for $\mu=1$), Table~\ref{TableAvamu2}  (for $\mu=2$) and Table~\ref{TableAvamu3}  (for $\mu=3$ and $\mu=4$).
\begin{table}[h]
\begin{center}
\begin{tabular}{|l||*{11}{c|}}
\hline
\backslashbox{$\varepsilon$}{$k$}&\makebox[1.5em]{1}&\makebox[1.5em]{2}&\makebox[1.5em]{3}
&\makebox[1.5em]{4}&\makebox[1.5em]{5}&\makebox[1.5em]{6}&\makebox[1.5em]{7}&\makebox[1.5em]{8}\\\hline\hline
1 &1/4&1/8&1/16&1/32&1/64&1/128&1/256&1/512\\\hline
2& 1/8&1/8& 3/32&5/64& 1/16& 13/256&21/512&17/512\\\hline
3 &1/16& 1/16& 1/16& 7/128& 13/256& 3/64& 11/256& 81/2048\\\hline
4&1/32&1/32& 1/32& 1/32& 15/512& 29/1024& 7/256& 27/1024\\\hline
5&1/64&1/64&1/64&1/64&1/64& 31/2048& 61/4096&15/1024\\\hline
\end{tabular}
\end{center}
\caption{For $\mu=1$ and for $\varepsilon=1,2,3,4,5$, the  probabilities $P[{L^{*}_{1,\varepsilon}}=k]$ that the full avalanche length takes values
 $k=1,...,8.$}\label{TableAvamu1}
\end{table}

\begin{table}[h]
\begin{center}
\begin{tabular}{|l||*{8}{c|}}
\hline
\backslashbox{$\varepsilon$}{$k$}&\makebox[1.5em]{1}&\makebox[1.5em]{2}&\makebox[1.5em]{3}
&\makebox[1.5em]{4}&\makebox[1.5em]{5}&\makebox[1.5em]{6}&\makebox[1.5em]{7}&\makebox[1.5em]{8}\\\hline\hline
1 &1/4&1/8&1/16&1/32&1/64&1/128&1/256&1/512\\\hline
2&1/4&1/8&1/16&1/32&1/64&1/128&1/256&1/512\\\hline
3 &3/16&3/32&3/32& 9/128& 3/64& 9/256& 27/1024& 39/2048\\\hline
4&5/32& 5/64& 5/64& 5/64& 15/256& 45/1024& 75/2048&125/4096\\\hline
5&1/8&1/16&1/16& 1/16& 1/16& 13/256& 21/512& 37/1024\\\hline
\end{tabular}
\end{center}
\caption{For $\mu=2$ and for $\varepsilon=1,2,3,4,5$, the  probabilities $P[{L^{*}_{2,\varepsilon}}=k]$ that the full avalanche length takes values
 $k=1,...,8.$}\label{TableAvamu2}
\end{table}

\begin{table}[h]
\begin{center}
\begin{tabular}{|l||*{8}{c|}}
\hline
\backslashbox{$\varepsilon$}{$k$}&\makebox[1.5em]{1}&\makebox[1.5em]{2}&\makebox[1.5em]{3}
&\makebox[1.5em]{4}&\makebox[1.5em]{5}&\makebox[1.5em]{6}&\makebox[1.5em]{7}&\makebox[1.5em]{8}\\\hline\hline
1 &1/4&1/8&1/16&1/32&1/64&1/128&1/256&1/512\\\hline
2&1/4&1/8&1/16&1/32&1/64&1/128&1/256&1/512\\\hline
3 &3/16&3/32&3/32& 9/128& 3/64& 9/256& 27/1024& 39/2048\\\hline
4&3/16&3/32&3/32& 9/128& 3/64& 9/256& 27/1024& 39/2048\\\hline
5&5/32& 5/64&5/64& 15/256& 15/256& 25/512& 75/2048& 125/4096\\\hline
\end{tabular}
\end{center}
\caption{For $\mu=3,4$ and for $\varepsilon=1, 2,3,4,5$, the  probabilities $P[{L^{*}_{3,\varepsilon}}=k]$ that the full avalanche length takes values
 $k=1,..., 8.$}\label{TableAvamu3}
\end{table}

\subsection{Limit results for the full avalanche length}
\label{subsec:limit-full}
We consider the Brownian limits
\begin{equation}
\left\{  
\frac{1}{\sqrt{n}}S_{\lfloor nt\rfloor}:t\geq0\right\}
\rightarrow\left\{  W(t):t\geq0\right\}
\end{equation}
in distribution on the Skorohod space, see \cite[Theorem 14.1. of XIV,P.146]{billingsley1999convergence}. Linear interpolation on the
Wiener space works similarily  see \cite[Theorem 8.6. of VIII,P.90]{billingsley1999convergence}.

We note that time must be scaled linearily, $t\mapsto nt$, and space must be
scaled with the square-root $\mu\mapsto\sqrt{n}\mu$. Excursion length refers to
time, excursion depth to space.

Up to here we have omitted in the notation for the intratrading times the dependence on the parameter $\mu$. Now we must use $\mu$ as a variable in the limiting procedure, thus we write from now on $T_1^{(\mu)}$ instead of $T_1$, and for brevity $T_1^{(n\mu)}$ for the more precise $T_1^{(\lfloor n\mu\rfloor)}$.
\begin{proposition}\label{LStransformT1}
We have for the Laplace-Stieltjes transform
\begin{equation}
E[e^{-\frac{s}{n}T_1^{(\sqrt{n}\mu)}}]=
1-\sqrt{2s}\tanh\left(\mu\sqrt{2s}\right)\cdot n^{-\frac12}
+\mathcal{O}\left(n^{-1}\right), \quad n\to\infty.
\end{equation}
\end{proposition}
\begin{proof}
Slightly lengthy, but still elementary asymptotic expansions:\footnote{The detailed calculations can be performed by hand in a few pages but can be also performed (and checked) automatically with a computer algebra system.} using the expansion of $\sqrt{1-z^2}$ for $z \to 1$ obtain the expansion of $\lambda_1(z)$ namely:
$$\lambda_1(z)=1+\sqrt{2}(1-z)^{1/2}+(1-z)+\frac{3}{2\sqrt{2}}(1-z)^{3/2}+\mathcal{O}(1-z)^2.$$
Inserting this into the power series for $\log{(1+x)}$ at $x=0$ yields
$$\log\lambda_1(z)=\sqrt{2}(1-z)^{1/2}+\frac{5}{6\sqrt{2}}(1-z)^{3/2}+\mathcal{O}(1-z)^2,$$
and a similar calculation yields
$$\log\lambda_2(z)=-\sqrt{2}(1-z)^{1/2}-\frac{5}{6\sqrt{2}}(1-z)^{3/2}+\mathcal{O}(1-z)^2.$$
Now we expand $e^{-s/n}$ for $n\to \infty$ using the exponential series, and the elementary formula
$$\lambda_1(e^{-s/n})^{\mu\sqrt{n}}=\exp(\mu \sqrt{n}\log\lambda_1(e^{-s/n}))$$
and combine the asymptotic expansions, similarily for $\lambda_2$. If we consider enough terms to account for cancellations in the differences in (\ref{gfa}) and (\ref{gfc}) we obtain expansion for $A_{{\mu}\sqrt{n}}(e^{-\frac{s}{n}})$ which coincides with $B_{{\mu}\sqrt{n}}(e^{-\frac{s}{n}})$, and we find also the expansion of $C_{{\mu}\sqrt{n}}+(e^{-\frac{s}{n}})$ as $n \to \infty.$ Finally we get a qutient of exponentials, which can be expressed as hyperbolic tangent.
\end{proof}

The following limit results involve the hyperbolic tangent $\tanh$, the
hyperbolic cosecant $\csch$, the hyperbolic secant $\sech$ and the hyperbolic
cotangent $\coth$.

\begin{proposition}[Hyperbolic function table]
We have pointwise for $n\rightarrow\infty$
the asymptotic relations
\begin{eqnarray}
&&  E[e^{-\frac{s}{n}T_1^{\mu\sqrt{n}}}]=
1-\sqrt{2s}\tanh(\mu\sqrt{2s})\cdot n^{-\frac{1}{2}}
+\mathcal{O}(n^{-1}),\label{Hyp1}
\\
&&  E[e^{-\frac{s}{n}T_1^{\mu\sqrt{n}}};S(T_1^{\mu\sqrt{n}})>0]
=1-\sqrt{2s}\coth(\mu\sqrt{2s})\cdot n^{-\frac{1}{2}}
+\mathcal{O}(n^{-1}),\label{Hyp2}
\\
&&  E[e^{-\frac{s}{n}T_1^{\mu\sqrt{n}}};S(T_1^{\mu\sqrt{n}})\leq0]
=2\sqrt{2s}\csch(2\mu\sqrt{2s})
+\mathcal{O}(n^{-1/2}),\label{Hyp3}
\\
&&  E[e^{-\frac{s}{n}\tau_D({\mu\sqrt{n}})}]
=\sech(2\mu\sqrt{2s})^{2}
+\mathcal{O}(n^{-1/2}).\label{Hyp4}
\end{eqnarray}
\end{proposition}

\begin{proof}
Elementary asymptotic calculations are performed, similarly as in the proof of Proposition \ref{LStransformT1}. 
Clearly the entries (\ref{Hyp2}) and (\ref{Hyp3}), add up to (\ref{Hyp1}) by the law of total probability. This corresponds in the limit to the elementary identity
\begin{equation}
\tanh\left(  \frac{z}{2}\right)  =\coth(z)-\csch(z),
\end{equation}
see \cite[4.5.30, P.84]{AS}. 
\end{proof}

\begin{lemma}\label{lemma11}
Let \begin{equation}
h(x)=\frac{1}{\sqrt{2\pi}}\cdot x^{-3/2}+2\sum_{k\geq1}\left[  \frac{1}%
{\sqrt{2\pi}}\cdot x^{-3/2}-2\sqrt{\frac{2}{\pi}}k^{2}\mu^{2}\cdot
x^{-5/2}\right]  \exp\left(  -\frac{2k^{2}\mu^{2}}{x}\right)  ,
\end{equation}
then we have \begin{equation}
\int_{0}^{\infty}(1-e^{-\lambda x})h(x)dx=\sqrt{2\lambda}\tanh\left(  \mu
\sqrt{2\lambda}\right)  .
\end{equation}
\end{lemma}
\begin{proof}
We justify termwise integration and sum the resulting series.
\end{proof}

\begin{theorem}The Laplace transform of the scaled full avalanche length for the simple symmetric random walk satisfies:
\begin{equation} \label{eqT3}
\lim_{n\to\infty} E[e^{-\frac{\lambda}{n}L_{\mu \sqrt{n},\varepsilon n}^*}]= \frac{\int_{\varepsilon}^{\infty}h(x)dx}{\int_{0}^{\varepsilon
}(1-e^{-\lambda x})h(x)dx+\int_{\varepsilon}^{\infty}h(x)dx}.
\end{equation}
\end{theorem}

\begin{proof}
We combine the limit results for the first trade in Proposition \ref{LStransformT1}, Lemma \ref{lemma11} and Lemma \ref{lem:vaglim}.
\end{proof}

\begin{remark}
Integrating the series for $h$ term by term, we find series representations for the integrals in the numerator and denominator in (\ref{eqT3}). The terms for $\int_{0} ^{\varepsilon}(1-e^{-\lambda x})h(x)dx$ can be expressed in terms of the error and complementary error functions. Termwise integration of $\int_{\varepsilon}^{\infty}h(x)dx$ yields thet there exists the series representation which can be recognize as a series representation of a  Jacobi theta function. 
We have 
\begin{equation}
\lim_{n\rightarrow\infty}P\left[  \frac{1}{n}T_{1}^{(\mu\sqrt{n})}%
>\varepsilon\right]  =\vartheta_{4}(0,e).
\end{equation}
\end{remark}

\newpage
\section{An initially empty order book}\label{SectionInitiallyEmpty}
Let us assume that we start with an order book that is
initially empty, i.e.
\begin{equation}
V(0,u)=0, \quad u\in\mathbb{Z}. 
\end{equation}

Furthermore, as in the case when the initially full limit order book is considered, the dynamics involves with respect to the equation (\ref{dynamics}). In order to be precise, specifically trading times $\{\widetilde{\tau}_{i}:i\geq0\}$ and intertrading times $(\widetilde{{T}}_{i})_{i\geq1}$ corresponding to the model when the limit order book is initially empty are  defined by

\begin{equation}
{\widetilde{\tau}}_{0}=0,\quad{\widetilde{\tau}}_{i}=\inf\{n>\widetilde{\tau}_{i-1}:V(n,S_{n})>0\},\quad  \widetilde{T}_{i}
={\widetilde{\tau}}_{i}-{\widetilde{\tau}}_{i-1},\quad i\geq1.
\end{equation}

The first trade (either Type~I or Type~II trade) will occur at smallest time $\widetilde{T}_1$ for which it is satisfied: 
$$|\min{(S(0),S(1),...,S(\widetilde{T}_1-1))}-S(\widetilde{T}_1)|=\mu.$$

\begin{remark}
Due to the fact that the limit order book model dynamics follow (\ref{dynamics}), note that if we start with an initially empty order book after occurring the first trade (either Type~I or Type~II trade) the analysis of the avalanche length is equivalent as we have started with initially full order book.
\end{remark}

\begin{figure}[h]
\begin{center}
\includegraphics[scale=1]{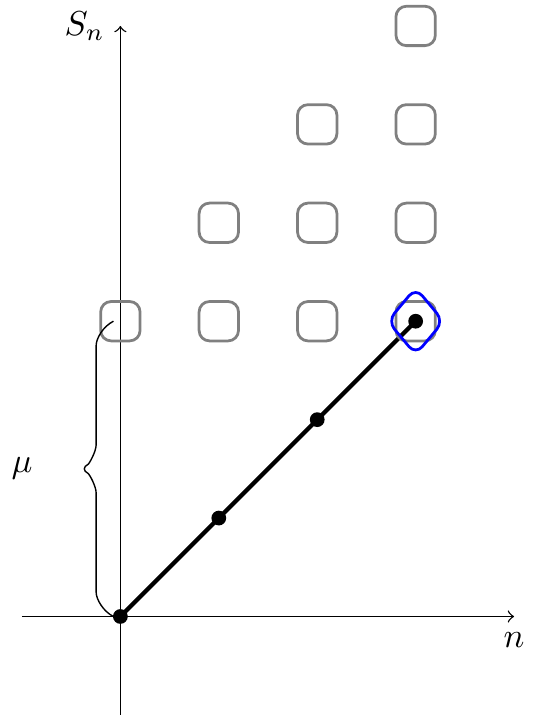}
\caption{Example of the path that depicts the trivial case for the first Type~I trade. }
\label{figIntTI}
\end{center}
\end{figure}

\subsection{The first Type~I trade in an initially empty book}
If the first Type~I trade in an initially empty book occurs at the price level $k$, for $k \in \{{1,2,...,\mu}\}$, then  the path $(S(0),S(1),...,S(\widetilde{T_1}))$ goes down by $\mu-k$ steps, but not below, and then goes up to the level $k$, as illustrated in Figure \ref{figIntTI2}. 
The trivial case when $k=\mu$ is depicted at Figure \ref{figIntTI}, the new order is placed at the start at distance $S_0+\mu=\mu$ and the first trade is occured at level $\mu$, after exactly $\mu$ steps up.

For $n\geq 1$, $\mu\ge 1$ and $k \in \{{1,2,...,\mu}\}$ we introduce the paths:

\begin{equation}\label{Gpath}
\mathscr{G}_{n,\mu,k}=\left\{s\in U_{n}:\mbox{$s_0=0$,
$k-\mu<s_j < \mu$ for $1\leq j\leq n-1$ and $s_n=k-\mu$}\right\},
\end{equation}

\begin{equation}\label{Fpath}
\mathscr{F}_{n, \mu,k}=\left\{s\in U_{n}:\mbox{$s_0=k-\mu$,
$k-\mu-1<s_j < k$ for $1\leq j\leq n-1$ and $s_n=k$}\right\}.
\end{equation}

\begin{figure}[h]
\begin{center}
\includegraphics[scale=1]{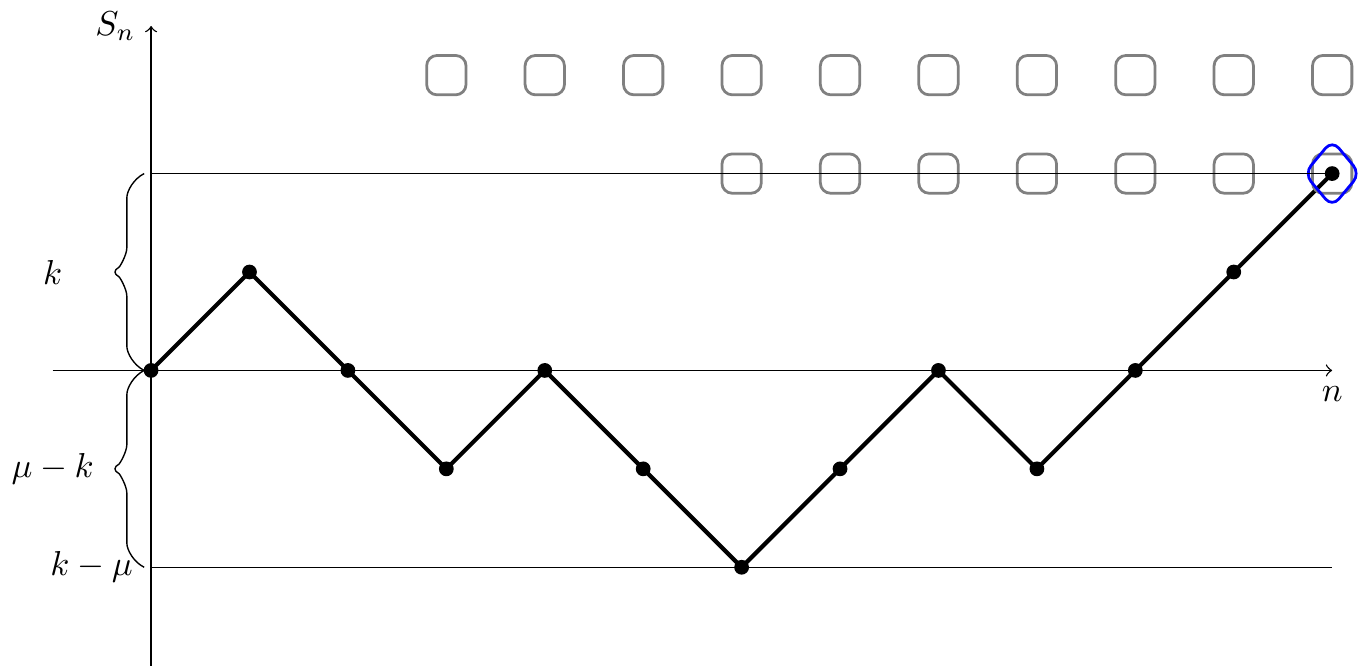}
\caption{The path of the first Type~I trade in the initially empty order book can be seen as a concatenation of two paths: first one which is starting at 0, forbidden level is $\mu$ and absorbing at level $k-\mu$; second path which is starting at $k-\mu$,  forbidden level is  $k-\mu-1$ and absorbing at level $k$, where $k \in \{{1,2,...,\mu}\}$}.
\label{figIntTI2}
\end{center}
\end{figure}

Furthermore, define the corresponding sets with arbitrary, finite path length by:
\begin{equation}
\mathscr{G}_{\mu,k}=\bigcup_{n\geq1}\mathscr{G}_{n,\mu,k},\quad\mathscr{F}_{\mu,k}=\bigcup_{n\geq1}\mathscr{F}_{n,\mu,k}
\end{equation}

\begin{lemma}\label{FirstTIempty}
For $k \in \{{1,2,...,\mu}\}$ denote by $G_{\mu,k}, F_{\mu,k}$ 
the probability generating functions for the combinatorial classes $\mathscr{G}_{\mu,k}, 
\mathscr{F}_{\mu,k}$ respectievly. We have
\begin{eqnarray}
G_{\mu,k}(s)&=&\frac{\lambda_{1}(s)^{\mu}-\lambda_{2}(s)^{\mu}}{\lambda
_{1}(s)^{2\mu-k}-\lambda_{2}(s)^{2\mu-k}},\label{gfg}\\
F_{\mu,k}(s)&=&\frac{\lambda_{1}(s)-\lambda_{2}(s)}{\lambda
_{1}(s)^{\mu+1}-\lambda_{2}(s)^{\mu+1}},\label{gff}
\end{eqnarray}
\end{lemma}
\begin{proof}
	
	The generating function $G_{\mu,k}$ is dervied following the formula \cite[XIV.4, (4.11), P.351]{Fel1}, by shifting a path from $\mathscr{G}_{n,\mu,k}$ by $\mu-k$. Thus, we obtain a path corresponding to absorption at zero
at the $n$-th trial in a symmetric game of gamblers ruin with initial position $\mu-k$ and absorbing barriers
at $0$ and $2\mu-k$.  The
formula \cite[XIV.4, (4.11), P.351]{Fel1} for $p=1/2$, $q=1/2$, $z=\mu-k$, $a=2\mu-k$ yields (\ref{gfg}).

The generating function $F_{\mu}$ is derived following the formula \cite[XIV.4, (4.12), P.351]{Fel1} by shifting a path from $\mathscr{F}_{n,\mu,k}$ by $\mu-k+1$.  Therefore, the formula is obtained for $p=1/2$, $q=1/2$, $z=1$, $a=\mu+1$. 

\end{proof}

Since the generating function $F_{\mu,k}$ does not depend on $k$,  from now on we can omit subscript $k$ and write $F_{\mu}.$

\subsection{The first Type~II trade in initially empty book}
The path of the first Type~II trade in the initially empty order book needs to down by $\mu$ or more, and then to go up by $\mu$ steps.
Therefore, if the first trade in initially empty order book is a Type~II trade, then the path $(S(0),\ldots,S(\widetilde{T_1}))$ is a concatenation of $K\geq1$ elements from $\mathscr{B}_\mu$ and
 one element from $\mathscr{C}_\mu$. The proof is same as in the Proposition~\ref{prop:decomp}.
 
 \subsection{The first trade in initially empty book}
 \begin{proposition}
If the order book is initially empty, the generating function of the time to the first trade is given by
\begin{equation}
\label{pgfTempty}
E[z^{\widetilde{T_1}}]=\sum_{k=o}^{\mu}G_{\mu-k}(z)F_{\mu}(z)+\frac{B_{\mu}(z)}{1-B_{\mu}(z)}\cdot C_{\mu}(z).
\end{equation}
\end{proposition}
\begin{proof}
Let us define the set $\widetilde{\mathscr{T}}_{\mu,n}$ of all paths where the first trade in an initially empty order book occurs at step $n$.
and $\widetilde{\mathscr{T}}_\mu=\bigcup_{n\geq1}\widetilde{\mathscr{T}}_{\mu,n}$.
This set contains elements that are either Type~I trades  or Type~II trades. By Lemma~\ref{FirstTIempty} the path of the Type~I trade is the concatenation of an element from the class $\mathscr{G}_{\mu,k}$ and an element from the class $\mathscr{F}_{\mu,k}$.  The path of the Type~II trade corresponds bijectivly to the Cartesian product
of the set of finite sequences of length $1$ or more of elements in $\mathscr{B}_\mu$ times
the class $\mathscr{C}_\mu$), see Proposition~\ref{prop:decomp}. Using the symbolic notation from \cite[Section~I]{FS} this is written more clearly as
\begin{equation}
\widetilde{\mathscr{T}}_\mu=\sum_{k=o}^{\mu}\mathscr{G}_{\mu-k}(z)\mathscr{F}_{\mu}(z)\cup
\SEQ_{\geq1}(\mathscr{B}_\mu)\times\mathscr{C}_\mu.
\end{equation}
From \cite[Theorem~I.1, P.27]{FS} and the section on {\em restricted constructions}, in particular
$\SEQ_{\geq k}$, in \cite[P.30]{FS},  the probability generating function (\ref{pgfTempty}) is obtained.
\end{proof}

\section{Auxiliary material}

\subsection{Some results of excursion theory}
\label{sec:excursion}
We recall some results of excursion theory, in particular we refer to Revuz and Yor\cite[XII.2, P.480]{RY}. Let $\left(
U,\mathcal{U}\right)  $  be the measurable space of Brownian excursions, and let
$\left(  e_{t},t>0\right)  $ be the excursion process. 
Since the 
Define by $U_{\delta}=U\cup\left\{  \delta\right\}  $ the set enhanced  by the
zero-excursion (which is set equal to $\delta$) on the set where the local
time at zero is strictly increasing. Further, this set is equipped with the $\sigma
$-algebra $\mathcal{U}_{\delta}=\sigma\left(  \mathcal{U},\left\{
\delta\right\}  \right)  $. 

For a measurable subset $\Gamma$ of $\mathcal{U}%
_{\delta}$, the function
\begin{equation}
N_{t}^{\Gamma}\left(  \omega\right)  =\sum_{0<u\leq t}\mathbf{1}_{\Gamma
}\left(  e_{u}\left(  \omega\right)  \right)  
\end{equation}
is measurable.

The \emph{Ito measure} $n$ is the $\sigma$-finite measure defined on
$\mathcal{U}$ by%
\begin{equation}
n\left(  \Gamma\right)  :=E\left[  N_{1}^{\Gamma}\right]
\end{equation}
and extended to $\mathcal{U}_{\delta}$ by $n\left(  \delta\right)=0$.

Since these functions graphs are
either entirely above or below the $t$-axis, we denote by $U^+$ and $U^-$  the
subsets of the set U. Further, $n^+$ and $n^-$ are upper and lower Ito measures, i.e. restrictions of $n$ to $U^+$ and $U^-$ respectievly.

It turns out that the excursion process is a Poisson Point Process, and hence
the Ito measure is its characteristic measure. An important consequence of
this is the Master Formula, see \cite[XII, Proposition1.10]{RY} for a
general version, which states that for a positive $\mathcal{U}_{\delta}%
$-measurable function $H$ defined on $U_{\delta}$ we have
\begin{equation}\label{MasterFormula}
E\left[  H\left(  e(\omega\right)  )\right]  =\int_{U}H(u)\,n(du).
\end{equation}
\subsection{On convergence in distribution and vague convergence}
\begin{lemma}
\label{lem:vaglim}
Suppose we are given a real number $\varepsilon>0$,
a sequence of probability measures $(\mu_n)_{n\geq1}$ on $\mathcal{B}((0,\infty))$ and
a non-negative measure $\nu$ on $\mathcal{B}((0,\infty))$, not the zero measure, such that
\begin{equation}
\int_0^\infty(1\wedge x)\nu(dx)<\infty.\label{wedgenu}
\end{equation}
If $\epsilon$ is a continuity point for $\nu$ and
\begin{equation}\label{asymu}
\int_0^\infty e^{-sx}\mu_n(dx)=1-\hat\nu(s)\cdot n^{-1/2}+\mathcal{O}(n^{-1})\quad n\to\infty
\end{equation}
pointwise for all $s\geq0$, where
\begin{equation}
\hat\nu(s)=\int_0^\infty(1-e^{-sx})\nu(dx),\label{hatnu}
\end{equation}
then we have
\begin{equation}
\lim_{n\to\infty}n^{1/2}\label{lapmu}
\int_0^\varepsilon(1-e^{-sx})\mu_n(dx)=
\int_0^\varepsilon(1-e^{-sx})\nu(dx)
\end{equation}
and
\begin{equation}\label{resla}
\lim_{n\to\infty}n^{1/2}
\int_\varepsilon^\infty\mu_n(dx)=
\int_\varepsilon^\infty\nu(dx)
\end{equation}
for all $\lambda\geq0$.
\end{lemma}
\begin{proof}
Assumption (\ref{wedgenu}) implies that the integral in (\ref{hatnu}) is finite for all $s\geq0$.
Let us denote the integral on the left hand side of (\ref{lapmu}), which is simply the Laplace transform 
of $\mu_n$ by $\tilde\mu(s)$. The case $\lambda=0$ is trivial, so fix $\lambda>0$.

Since we assumed that $\mu_n$ lives on $(0,\infty)$ we have $0\leq\tilde\mu_n(\lambda)<1$ for 
$\lambda>0$.
Since we assumed that $\nu$ is not the zero measure, we have also $\hat\nu(\lambda)>0$.
Let us define another new measure $\nu_\lambda$ by
\begin{equation}
\nu_\lambda(dx)=\frac{1-e^{-\lambda x}}{\hat\nu(\lambda)}\nu_n(dx).
\end{equation}
Note that $\nu_\lambda$ is a probability measure. Denote by $\hat\nu_\lambda(s)$ its Laplace transform.
The asymptotics (\ref{asymu}) imply 
\begin{equation}
\lim_{n\to\infty}\tilde\gamma_n(s)=\hat\nu_\lambda(s)
\end{equation}
pointwise for all $s\geq0$. By the continuity theorem for Laplace transforms, e.g. \cite[Theorem~XIII.1.2a, P.433]{Fel2},
it follows that $\gamma_n\to\nu_\lambda$ weakly as $n\to\infty$. As $\varepsilon$ is also a continuity point
for the limit distribution $\nu_\lambda$, it follows 
\begin{equation}
\lim_{n\to\infty}\label{limdf}
\int_0^\varepsilon\gamma_n(dx)=
\int_0^\varepsilon\nu_\lambda(dx).
\end{equation}
Relation (\ref{asymu}) implies also
\begin{equation}
\lim_{n\to\infty}n^{1/2}(1-\tilde\mu_n(\lambda))=\hat\nu(\lambda).\label{limnu}
\end{equation}
Combining (\ref{limdf}) and (\ref{limnu}) yields (\ref{resla}).
No consider the function
\begin{equation}
f(x)=\frac{\hat\nu(\lambda)}{1-e^{-\lambda x}}I_{x>\varepsilon},\quad x>0.
\end{equation}
It is bounded and continuous except for the point $x=\varepsilon$. By the continuous mapping
theorem, as formulated for example in \cite[Theorem~3.2.4, P.101]{Dur} we get
\begin{equation}
\lim_{n\to\infty}
\int_\varepsilon^\infty f(x)\gamma_n(dx)=
\int_\varepsilon^\infty f(x)\nu_\lambda(dx).
\end{equation}
Using the definitions of $\gamma_n$ and $\nu_\lambda$ and (\ref{limnu}) we obtain (\ref{lapmu}).
\end{proof}
\begin{proposition}\label{prop:vaglim}
In the setting of Lemma~\ref{lem:vaglim} we have
\begin{equation}
\lim_{n\to\infty}n^{1/2}\mu_n=\nu
\end{equation}
vaguely on $(0,\infty)$.
\end{proposition}
\begin{proof}
Suppose $f$ is a continuous function with compact support in $(0,\infty)$.
Define $g(x)=\hat\nu(\lambda)f(x)/(1-e^{-\lambda x}$ for $x>0$.
Since $f$ vanishes is some neighbourhood of $x=0$, it follows that $g$ is a bounded continuous function.
Thus we have from weak convergence
\begin{equation}
\lim_{n\to\infty}
\int_0^\infty g(x)\gamma_n(dx)=
\int_0^\infty g(x)\nu(dx),
\end{equation}
which can be rewritten with $(\ref{limnu})$ as
\begin{equation}
\lim_{n\to\infty}
\int_0^\infty f(x)n^{1/2}\mu_n(dx)=
\int_0^\infty f(x)\nu(dx),
\end{equation}
thus showing vague convergence.
\end{proof}

\section{Acknowledgment}
The authors would like to thank Thorsten Rheinl\"ander and Sabine Sporer,
whose work motivated the present paper.


\begin{thebibliography}{AAC{\etalchar{+}}16}
	
	\bibitem[AAC{\etalchar{+}}16]{abergel2016limit}
	Fr{\'e}d{\'e}ric Abergel, Marouane Anane, Anirban Chakraborti, Aymen Jedidi,
	and Ioane~Muni Toke.
	\newblock {\em Limit order books}.
	\newblock Cambridge University Press, 2016.
	
	\bibitem[AJ13]{abergel2013mathematical}
	Fr{\'e}d{\'e}ric Abergel and Aymen Jedidi.
	\newblock A mathematical approach to order book modeling.
	\newblock {\em International Journal of Theoretical and Applied Finance},
	16(05):1350025, 2013.
	
	\bibitem[Ald98]{Ald1998}
	David~J. Aldous.
	\newblock Brownian excursion conditioned on its local time.
	\newblock {\em Electronic Communications in Probability}, 3:79--90
	(electronic), 1998.
	
	\bibitem[AS64]{AS}
	Milton Abramowitz and Irene~A. Stegun.
	\newblock {\em Handbook of mathematical functions with formulas, graphs, and
		mathematical tables}, volume~55 of {\em National Bureau of Standards Applied
		Mathematics Series}.
	\newblock For sale by the Superintendent of Documents, U.S. Government Printing
	Office, Washington, D.C., 1964.
	
	\bibitem[BCP03]{BCP2003}
	Jean Bertoin, Lo{\"{\i}}c Chaumont, and Jim Pitman.
	\newblock Path transformations of first passage bridges.
	\newblock {\em Electronic Communications in Probability}, 8:155--166, 2003.
	
	\bibitem[BHQ14]{BHQ}
	Christian Bayer, Ulrich Horst, and Jinniao Qiu.
	\newblock {A Functional Limit Theorem for Limit Order Books}.
	\newblock {\em ArXiv e-prints}, May 2014.
	
	\bibitem[Bil99]{billingsley1999convergence}
	Patrick Billingsley.
	\newblock Convergence of probability measures john wiley \& sons.
	\newblock {\em INC, New York}, 2(2.4), 1999.
	
	\bibitem[CdL13]{CL}
	Rama Cont and Adrien de~Larrard.
	\newblock Price dynamics in a {M}arkovian limit order market.
	\newblock {\em SIAM Journal on Financial Mathematics}, 4(1):1--25, 2013.
	
	\bibitem[CH03]{CH2003}
	Endre Cs{\'a}ki and Yueyun Hu.
	\newblock Lengths and heights of random walk excursions.
	\newblock In {\em Discrete random walks ({P}aris, 2003)}, Discrete Math. Theor.
	Comput. Sci. Proc., AC, pages 45--52. Assoc. Discrete Math. Theor. Comput.
	Sci., Nancy, 2003.
	
	\bibitem[CH04]{CH2004}
	Endre Cs{\'a}ki and Yueyun Hu.
	\newblock Invariance principles for ranked excursion lengths and heights.
	\newblock {\em Electronic Communications in Probability}, 9:14--21, 2004.
	
	\bibitem[CM86]{CM1986}
	E.~Cs{\'a}ki and S.~G. Mohanty.
	\newblock Some joint distributions for conditional random walks.
	\newblock {\em The Canadian Journal of Statistics}, 14(1):19--28, 1986.
	
	\bibitem[CR92]{CR1992}
	Mikl{\'o}s Cs{\"o}rg{\H{o}} and P{\'a}l R{\'e}v{\'e}sz.
	\newblock Long random walk excursions and local time.
	\newblock {\em Stochastic Processes and their Applications}, 41(2):181--190,
	1992.
	
	\bibitem[Cs{\'a}94]{Csa1994}
	Endre Cs{\'a}ki.
	\newblock Some joint distributions in {B}ernoulli excursions.
	\newblock {\em Journal of Applied Probability}, 31A:239--250, 1994.
	\newblock Studies in applied probability.
	
	\bibitem[Cs{\'a}96]{Csa1996}
	Endre Cs{\'a}ki.
	\newblock On some results for {B}ernoulli excursions.
	\newblock {\em Journal of Statistical Planning and Inference}, 54(1):45--54,
	1996.
	
	\bibitem[DRR13]{dLRR}
	Sylvain Delattre, Christian~Y. Robert, and Mathieu Rosenbaum.
	\newblock Estimating the efficient price from the order flow: a {B}rownian
	{C}ox process approach.
	\newblock {\em Stochastic Processes and their Applications}, 123(7):2603--2619,
	2013.
	
	\bibitem[DT96]{DT1996}
	I.~M. Davies and A.~Truman.
	\newblock Discrete random walks and excursions.
	\newblock In {\em Stochastic analysis and applications ({P}owys, 1995)}, pages
	165--175. World Sci. Publ., River Edge, NJ, 1996.
	
	\bibitem[Dur10]{Dur}
	Rick Durrett.
	\newblock {\em Probability: theory and examples}.
	\newblock Cambridge University Press, Cambridge, fourth edition, 2010.
	
	\bibitem[dW13]{DDW}
	L.~Dudok de~Wit.
	\newblock {\em Liquidity risks based on the limit order book}.
	\newblock Master thesis, Vienna University of Technology, 2013.
	
	\bibitem[DW15]{DW}
	A.~Dassios and S.~Wu.
	\newblock Two-side {Parisian} option with single barrier.
	\newblock To appear in Mathematical Finance, 2015.
	
	\bibitem[Fel68]{Fel1}
	William Feller.
	\newblock {\em An introduction to probability theory and its applications.
		{V}ol. {I}}.
	\newblock Third edition. John Wiley \& Sons Inc., New York, 1968.
	
	\bibitem[Fel71]{Fel2}
	William Feller.
	\newblock {\em An introduction to probability theory and its applications.
		{V}ol. {II}}.
	\newblock Second edition. John Wiley \& Sons Inc., New York, 1971.
	
	\bibitem[F{\"o}l94]{Foe1994}
	Ant{\'o}nia F{\"o}ldes.
	\newblock Runs and excursions.
	\newblock In {\em Runs and patterns in probability: selected papers}, volume
	283 of {\em Math. Appl.}, pages 243--251. Kluwer Acad. Publ., Dordrecht,
	1994.
	
	\bibitem[FS09]{FS}
	Philippe Flajolet and Robert Sedgewick.
	\newblock {\em Analytic combinatorics}.
	\newblock Cambridge University Press, Cambridge, 2009.
	
	\bibitem[HK17]{horst2017weak}
	Ulrich Horst and D{\"o}rte Kreher.
	\newblock A weak law of large numbers for a limit order book model with fully
	state dependent order dynamics.
	\newblock {\em SIAM Journal on Financial Mathematics}, 8(1):314--343, 2017.
	
	\bibitem[HR14]{HWP}
	Friedrich Hubalek and Throsten Rheinl{\"a}nder.
	\newblock On the distribution of the simplified avalanche length.
	\newblock Working paper, 2014.
	
	\bibitem[KKST17]{kirilenko2017flash}
	Andrei Kirilenko, Albert~S Kyle, Mehrdad Samadi, and Tugkan Tuzun.
	\newblock The flash crash: High-frequency trading in an electronic market.
	\newblock {\em The Journal of Finance}, 72(3):967--998, 2017.
	
	\bibitem[Kru03]{K}
	{\L}ukasz Kruk.
	\newblock Functional limit theorems for a simple auction.
	\newblock {\em Mathematics of Operations Research}, 28(4):716--751, 2003.
	
	\bibitem[LH07]{LH2007}
	Andreas Lindell and Lars Holst.
	\newblock Distributions of the longest excursions in a tied down simple random
	walk and in a {B}rownian bridge.
	\newblock {\em Journal of Applied Probability}, 44(4):1056--1067, 2007.
	
	\bibitem[LM07]{LM2007}
	Jean-Maxime Labarbe and Jean-Fran{\c{c}}ois Marckert.
	\newblock Asymptotics of {B}ernoulli random walks, bridges, excursions and
	meanders with a given number of peaks.
	\newblock {\em Electronic Journal of Probability}, 12:no. 9, 229--261, 2007.
	
	\bibitem[OLBC10]{OLBC}
	Frank~WJ Olver, Daniel~W Lozier, Ronald~F Boisvert, and Charles~W Clark.
	\newblock {\em NIST handbook of mathematical functions hardback and CD-ROM}.
	\newblock Cambridge university press, 2010.
	
	\bibitem[PW14]{PW2014}
	Mihael Perman and Jon~A. Wellner.
	\newblock An excursion approach to maxima of the {B}rownian bridge.
	\newblock {\em Stochastic Processes and their Applications}, 124(9):3106--3120,
	2014.
	
	\bibitem[Ric13]{R}
	F.~Riccardi.
	\newblock {\em Stochastic Models for the Limit Order Book}.
	\newblock Mphil thesis, London School of Economics, 2013.
	
	\bibitem[RY99]{RY}
	Daniel Revuz and Marc Yor.
	\newblock {\em Continuous martingales and {B}rownian motion}.
	\newblock Springer-Verlag, Berlin, third edition, 1999.
	\newblock (in particular Ch. XII: Excursion theory).
	
	\bibitem[SC06]{SC}
	M.~A. Stapleton and K.~Christensen.
	\newblock One-dimensional directed sandpile models and the area under a
	{B}rownian curve.
	\newblock {\em Journal of Physics. A. Mathematical and General},
	39(29):9107--9126, 2006.
	
	\bibitem[SD88]{SD1988}
	Laurel Smith and Persi Diaconis.
	\newblock Honest {B}ernoulli excursions.
	\newblock {\em Journal of Applied Probability}, 25(3):464--477, 1988.
	
	\bibitem[Spo14]{Sp}
	Sabine Sporer.
	\newblock {\em {Verteilungen} im {Rahmen} des {Limit-Order-Buches}}.
	\newblock Master thesis, Vienna University of Technology, 2014.
	
	\bibitem[Tak95]{Tak1995}
	Lajos Tak{\'a}cs.
	\newblock Brownian local times.
	\newblock {\em Journal of Applied Mathematics and Stochastic Analysis},
	8(3):209--232, 1995.
	
	\bibitem[Tak99]{Tak1999}
	Lajos Tak{\'a}cs.
	\newblock The distribution of the sojourn time for the {B}rownian excursion.
	\newblock {\em Methodology and Computing in Applied Probability}, 1(1):7--28,
	1999.
	
	\bibitem[Ver79]{Ver1979}
	Wim Vervaat.
	\newblock A relation between {B}rownian bridge and {B}rownian excursion.
	\newblock {\em The Annals of Probability}, 7(1):143--149, 1979.
	
	\bibitem[VZFR14]{VZFR}
	Marcela Valenzuela, Ilknur Zer, Piotr Fryzlewicz, and Thorsten Rheinl{\"a}nder.
	\newblock Relative liquidity and future volatility.
	\newblock Technical Report~45, Federal Reserve Board, Washington, D.C., 2014.
	
\end{thebibliography}

\bibliographystyle{alpha}

\newcommand{\etalchar}[1]{$^{#1}$}

\end{document}